\documentclass[11pt]{article}
\usepackage{amssymb,amsfonts,amsmath,latexsym,epsf,tikz,url}
\usepackage[usenames,dvipsnames]{pstricks}
\usepackage{pstricks-add}
\usepackage{epsfig}
\usepackage{pst-grad} 
\usepackage{pst-plot} 
\usepackage[space]{grffile} 
\usepackage{etoolbox} 
\makeatletter 
\patchcmd\Gread@eps{\@inputcheck#1 }{\@inputcheck"#1"\relax}{}{}
\makeatother

\newtheorem{theorem}{Theorem}[section]

\newtheorem{corollary}[theorem]{Corollary}
\newtheorem{observation}[theorem]{Observation}

\newtheorem{remark}[theorem]{Remark}

\newcommand{\proof}{\noindent{\bf Proof.\ }}
\newcommand{\qed}{\hfill $\square$\medskip}

\begin{document}

\title{On the number of isolate dominating sets of certain graphs}

\author{
Nima Ghanbari$^1$
\and
Saeid Alikhani$^{2,}$\footnote{Corresponding author}
}

\date{\today}

\maketitle
\begin{center}
$^1$Department of Informatics, University of Bergen, P.O. Box 7803, 5020 Bergen, Norway\\
$^2$Department of Mathematics, Yazd University, 89195-741, Yazd, Iran\\
\medskip
{\tt Nima.ghanbari@uib.no, alikhani@yazd.ac.ir}
\end{center}


\begin{abstract}
 
Let $G=(V,E)$ be a simple graph. A dominating set of $G$ is a subset $S\subseteq V$ such that every vertex not in $S$ is adjacent to at least one vertex in $S$.
The cardinality of a smallest dominating set of $G$, denoted by $\gamma(G)$, is the domination number of $G$. A dominating set $S$ is an isolate dominating set of $G$, if 
the induced subgraph $G[S]$ has at least one isolated vertex.  The isolate domination number, $\gamma_0(G)$, is the minimum cardinality of an isolate dominating set of $G$.  In this paper,  we  count the number of isolate dominating sets of some specific graphs.

\end{abstract}

\noindent{\bf Keywords:} isolate domination, isolate dominating set, corona, path, cycle.  

\medskip
\noindent{\bf AMS Subj.\ Class.}:05C15

\section{Introduction}

Let $G = (V,E)$ be a simple graph with $n$ vertices. Throughout this paper we consider only simple graphs.  A set $S\subseteq V(G)$ is a  dominating set if every vertex in $V(G)\backslash S$ is adjacent to at least one vertex in $S$.
The  domination number $\gamma(G)$ is the minimum cardinality of a dominating set in $G$. There are various domination numbers in the literature.
For a detailed treatment of domination theory, the reader is referred to \cite{domination}.

A dominating set $S$ of a graph $G$ with no isolated vertex is called a total dominating set of $G$ if the induced subgraph $G[S]$ has no isolated vertex. The total domination number $\gamma_t(G)$ is the minimum cardinality of a total dominating set of $G$.  A dominating set $S$ is an isolate dominating set of $G$, if 
the induced subgraph $G[S]$ has at least one isolated vertex.  The isolate domination number, $\gamma_0(G)$, is the minimum cardinality of an isolate dominating set of $G$.  A dominating set
with cardinality $\gamma(G)$  is  called a {\it $\gamma$-set}.  Also an isolate dominating set of $G$ of cardinality $\gamma_0(G)$ is called a $\gamma_0$-set of $G$.
The study of  isolate  domination in graphs was initiated by Hamid and Balamurugan \cite{Hamid}, and further studied in a number of papers (see for example  \cite{Nader,Hamid2}).

The concept of domination and related invariants have
been generalized in many ways.  Most of the papers published so far deal with structural
aspects of domination, trying to determine exact expressions for $\gamma(G)$  or some upper and/or lower bounds for it. There were no paper concerned with the 
enumerative side of the problem by 2008.

Regarding to enumerative side of dominating sets, domination polynomial of graph has introduced in \cite{saeid1}. The domination polynomial of graph $G$ is the  generating function for the number of dominating sets of  $G$, i.e., $D(G,x)=\sum_{ i=1}^{|V(G)|} d(G,i) x^{i}$ (see \cite{euro,saeid1}).   This  polynomial and its roots has been actively studied in recent
years (see for example \cite{Kot,Oboudi}). 
It is natural to count the number of another kind of dominating sets (\cite{utilitas,weakly,Doslic,DAM}).     
Let ${\cal D}_0(G,i)$ be the family of
isolate  dominating sets of a graph $G$ with cardinality $i$ and let
$d_0(G,i)=|{\cal D}_0(G,i)|$. The generating function for the number of isolate dominating sets of $G$ is $D_0(G,x)=\sum_{i=1}d_0(G,i)x^i$. 

The  corona of two graphs $G_1$ and $G_2$, is the graph
$G_1 \circ G_2$ formed from one copy of $G_1$ and $|V(G_1)|$ copies of $G_2$,
where the ith vertex of $G_1$ is adjacent to every vertex in the ith copy of $G_2$.
The corona $G\circ K_1$, in particular, is the graph constructed from a copy of $G$,
where for each vertex $v\in V(G)$, a new vertex $v'$ and a pendant edge $vv'$ are added.
The  join of two graphs $G_1$ and $G_2$, denoted by $G_1\vee G_2$,
is a graph with vertex set  $V(G_1)\cup V(G_2)$
and edge set $E(G_1)\cup E(G_2)\cup \{uv| u\in V(G_1)$ and $v\in V(G_2)\}$.

\medskip
In the next section, we study the number of isolate dominating (I.D.) sets of specific graphs, especially for paths and cycles. In Section 3, we consider the corona product of two graphs and investigate the number of their isolate dominating sets.

\section{The number of I.D. sets of paths and cycles}

In this section, we study the number of isolate dominating (I.D.) sets of certain graphs.
Let ${\cal D}_0(G,i)$ be the family of
isolate  dominating sets of a graph $G$ with cardinality $i$ and let
$d_0(G,i)=|{\cal D}_0(G,i)|$. 
The following easy theorem is about the number of isolate dominating sets of the complete graph $K_n$ and the star graph $K_{1,n}$:

\begin{observation}
	\begin{enumerate}
		\item [(i)] For any $n\in \mathbb{N}$, $d_0(K_n,1)=n$ and for $i\geq 2$, $d_0(K_n,i)=0$.
		
		\item[(ii)]   
		For  $n\in \mathbb{N}$, $d_0(K_{1,n},1)=d_0(K_{1,n},n)=1$ and for $2\leq i\leq n-1$, $d_0(K_{1,n},i)=0$.
			\end{enumerate}
	\end{observation}

The following theorem is about the number of isolate dominating sets of $G_1\vee G_2$:

\begin{theorem} 
	For any graphs $G_1$ and $G_2$ of order at least two,  
	$$d_0(G_1\vee G_2,i)= d_0(G,i)+d_0(G_2,i).$$
\end{theorem} 
\begin{proof} 
	By  the definition, every isolate dominating set $S_1$ of $G_1$ (or $S_2$ of $G_2$) is an isolated dominating set of $G_1\vee G_2$. Therefore we have the result. \qed
	\end{proof} 

  We consider the  path graph $P_n$ and cycle graph $C_n$. Note that $\gamma_0(P_n)=\gamma_0(C_n)=\lceil\frac{n}{3}\rceil$ (see \cite{Hamid}). We need the following theorem: 
	
	\begin{theorem}{\rm\cite{Saeid}}
	The number of dominating sets of path $P_n$ satisfies  the following recursive relation: 
	\[d(P_n,i)= d(P_{n-1},i-1)+d(P_{n-1},i-2)+d(P_{n-1},i-3).\]

	\end{theorem}
	
The following theorem gives the explicit formula for the number of dominating sets of $P_n$ (\cite{Llano}):

\begin{theorem} 
	For every $n\geq 1$, $d(P_n,k)=\displaystyle\sum_{m=0}^{\lfloor\frac{n-k}{2}\rfloor+1}{k-1 \choose n-k-m}{n-k-m+2 \choose m}.$ 
	\end{theorem} 
	
		\begin{figure}
		\begin{center}
			\psscalebox{0.56 0.56}
			{				\begin{pspicture}(0,-7.87)(12.97,3.69)
				\psdots[linecolor=black, dotsize=0.4](1.32,2.98)
				\psdots[linecolor=black, dotsize=0.4](2.12,2.98)
				\psdots[linecolor=black, dotsize=0.4](2.92,2.98)
				\psdots[linecolor=black, dotsize=0.4](3.72,2.98)
				\psdots[linecolor=black, dotsize=0.4](4.52,2.98)
				\psdots[linecolor=black, dotsize=0.4](5.32,2.98)
				\psline[linecolor=black, linewidth=0.08](0.92,1.78)(6.12,1.78)(6.12,1.78)
				\psline[linecolor=black, linewidth=0.08](0.92,0.98)(0.92,0.98)(6.12,0.98)(6.12,0.98)
				\psline[linecolor=black, linewidth=0.08](2.52,1.78)(2.52,0.98)(2.52,0.98)
				\psline[linecolor=black, linewidth=0.08](3.32,0.98)(3.32,0.18)(3.32,0.18)
				\psline[linecolor=black, linewidth=0.08](0.92,0.18)(6.12,0.18)(6.12,0.18)
				\psline[linecolor=black, linewidth=0.08](3.32,0.18)(3.32,-0.62)(3.32,-0.62)
				\psline[linecolor=black, linewidth=0.08](0.92,-0.62)(6.12,-0.62)(6.12,-0.62)
				\psdots[linecolor=black, dotsize=0.4](1.32,1.38)
				\psdots[linecolor=black, dotsize=0.4](2.12,0.58)
				\psdots[linecolor=black, dotsize=0.4](2.12,-0.22)
				\psdots[linecolor=black, dotsize=0.4](1.32,-0.22)
				\psdots[linecolor=black, dotstyle=o, dotsize=0.4, fillcolor=white](2.12,1.38)
				\psdots[linecolor=black, dotstyle=o, dotsize=0.4, fillcolor=white](2.92,0.58)
				\psdots[linecolor=black, dotstyle=o, dotsize=0.4, fillcolor=white](1.32,0.58)
				\psdots[linecolor=black, dotstyle=o, dotsize=0.4, fillcolor=white](2.92,-0.22)
				\psdots[linecolor=black, dotstyle=o, dotsize=0.4, fillcolor=white](3.72,-1.02)
				\psline[linecolor=black, linewidth=0.08](0.92,-2.22)(0.92,1.78)(0.92,1.78)
				\psline[linecolor=black, linewidth=0.08](0.92,-1.42)(6.12,-1.42)(6.12,-1.42)
				\psline[linecolor=black, linewidth=0.08](0.92,-2.22)(6.12,-2.22)(6.12,-2.22)
				\psdots[linecolor=black, dotstyle=o, dotsize=0.4, fillcolor=white](1.32,-1.02)
				\psdots[linecolor=black, dotsize=0.4](2.12,-1.02)
				\psdots[linecolor=black, dotsize=0.4](2.92,-1.02)
				\psline[linecolor=black, linewidth=0.08](4.12,-0.62)(4.12,-2.22)(4.12,-2.22)
				\psdots[linecolor=black, dotsize=0.4](1.32,-1.82)
				\psdots[linecolor=black, dotsize=0.4](2.12,-1.82)
				\psdots[linecolor=black, dotsize=0.4](2.92,-1.82)
				\psdots[linecolor=black, dotstyle=o, dotsize=0.4, fillcolor=white](3.72,-1.82)
				\psdots[linecolor=black, dotstyle=o, dotsize=0.4, fillcolor=white](4.52,-2.62)
				\psdots[linecolor=black, dotstyle=o, dotsize=0.4, fillcolor=white](1.32,-2.62)
				\psline[linecolor=black, linewidth=0.08](0.92,-1.42)(0.92,-3.82)(0.92,-3.82)
				\psline[linecolor=black, linewidth=0.08](0.92,-3.82)(6.12,-3.82)(6.12,-3.82)
				\psline[linecolor=black, linewidth=0.08](0.92,-3.02)(6.12,-3.02)(6.12,-3.02)
				\psdots[linecolor=black, dotsize=0.4](2.12,-2.62)
				\psdots[linecolor=black, dotsize=0.4](2.92,-2.62)
				\psdots[linecolor=black, dotsize=0.4](3.72,-2.62)
				\psline[linecolor=black, linewidth=0.08](4.92,-2.22)(4.92,-3.82)(4.92,-3.82)
				\psdots[linecolor=black, dotsize=0.4](1.32,-3.42)
				\psdots[linecolor=black, dotsize=0.4](2.12,-3.42)
				\psdots[linecolor=black, dotsize=0.4](2.92,-3.42)
				\psdots[linecolor=black, dotsize=0.4](3.72,-3.42)
				\psdots[linecolor=black, dotsize=0.4](2.12,-4.22)
				\psdots[linecolor=black, dotsize=0.4](2.92,-4.22)
				\psdots[linecolor=black, dotsize=0.4](3.72,-4.22)
				\psdots[linecolor=black, dotsize=0.4](4.52,-4.22)
				\psdots[linecolor=black, dotstyle=o, dotsize=0.4, fillcolor=white](4.52,-3.42)
				\psdots[linecolor=black, dotstyle=o, dotsize=0.4, fillcolor=white](5.32,-4.22)
				\psdots[linecolor=black, dotstyle=o, dotsize=0.4, fillcolor=white](1.32,-4.22)
				\psline[linecolor=black, linewidth=0.08](5.72,-3.82)(5.72,-4.62)(0.92,-4.62)(0.92,-3.82)(0.92,-3.82)
				\psline[linecolor=black, linewidth=0.08](5.32,-4.62)(6.12,-4.62)(6.12,-4.62)
				\psdots[linecolor=black, dotsize=0.1](2.12,-5.02)
				\psdots[linecolor=black, dotsize=0.1](2.12,-5.42)
				\psdots[linecolor=black, dotsize=0.1](2.12,-5.82)
				\psline[linecolor=black, linewidth=0.08](0.92,-4.62)(0.92,-6.22)(6.12,-6.22)(4.52,-6.22)
				\psline[linecolor=black, linewidth=0.08](0.92,-6.22)(0.92,-7.02)(6.12,-7.02)(6.12,-7.02)
				\psdots[linecolor=black, dotsize=0.4](2.12,-6.62)
				\psdots[linecolor=black, dotsize=0.4](2.92,-6.62)
				\psdots[linecolor=black, dotsize=0.4](3.72,-6.62)
				\psdots[linecolor=black, dotsize=0.4](4.52,-6.62)
				\psdots[linecolor=black, dotsize=0.4](5.32,-6.62)
				\psdots[linecolor=black, dotsize=0.4](2.12,-7.42)
				\psdots[linecolor=black, dotsize=0.4](2.92,-7.42)
				\psdots[linecolor=black, dotsize=0.4](3.72,-7.42)
				\psdots[linecolor=black, dotsize=0.4](4.52,-7.42)
				\psdots[linecolor=black, dotsize=0.4](5.32,-7.42)
				\psline[linecolor=black, linewidth=0.08](0.92,-7.02)(0.92,-7.82)(6.12,-7.82)(6.12,-7.82)
				\psdots[linecolor=black, dotsize=0.4](11.72,2.98)
				\psdots[linecolor=black, dotsize=0.4](12.52,2.98)
				\psline[linecolor=black, linewidth=0.08](6.12,1.78)(12.92,1.78)(12.92,-7.82)(6.12,-7.82)(6.12,-7.82)
				\psline[linecolor=black, linewidth=0.08](5.72,-7.02)(12.92,-7.02)(12.92,-7.02)
				\psline[linecolor=black, linewidth=0.08](6.12,-6.22)(12.92,-6.22)(12.92,-6.22)
				\psline[linecolor=black, linewidth=0.08](6.12,-4.62)(12.92,-4.62)(12.92,-4.62)
				\psline[linecolor=black, linewidth=0.08](6.12,-3.82)(12.92,-3.82)(12.92,-3.82)
				\psline[linecolor=black, linewidth=0.08](6.12,-3.02)(12.92,-3.02)(12.92,-3.02)
				\psline[linecolor=black, linewidth=0.08](6.12,-2.22)(12.92,-2.22)(12.92,-2.22)
				\psline[linecolor=black, linewidth=0.08](6.12,-1.42)(6.12,-1.42)(6.12,-1.42)(6.12,-1.42)(12.92,-1.42)(12.92,-1.42)
				\psline[linecolor=black, linewidth=0.08](6.12,-0.62)(12.92,-0.62)(12.92,-0.62)
				\psline[linecolor=black, linewidth=0.08](6.12,0.18)(12.92,0.18)(12.92,0.18)
				\psline[linecolor=black, linewidth=0.08](5.72,0.98)(12.92,0.98)(12.92,0.98)
				\rput[bl](1.14,3.38){${v_1}$}
				\rput[bl](1.96,3.4){${v_2}$}
				\rput[bl](2.7,3.4){${v_3}$}
				\rput[bl](3.52,3.42){${v_4}$}
				\rput[bl](4.38,3.44){${v_5}$}
				\rput[bl](5.16,3.42){${v_6}$}
				\rput[bl](12.38,3.26){${v_n}$}
				\rput[bl](11.4,3.22){${v_{n-1}}$}
				\psdots[linecolor=black, dotsize=0.1](6.12,-6.62)
				\psdots[linecolor=black, dotsize=0.1](6.52,-6.62)
				\psdots[linecolor=black, dotsize=0.1](6.92,-6.62)
				\psdots[linecolor=black, dotsize=0.1](6.12,-7.42)
				\psdots[linecolor=black, dotsize=0.1](6.52,-7.42)
				\psdots[linecolor=black, dotsize=0.1](6.92,-7.42)
				\psdots[linecolor=black, dotsize=0.4](7.72,-6.62)
				\psdots[linecolor=black, dotsize=0.4](7.72,-7.42)
				\psdots[linecolor=black, dotsize=0.4](8.52,-7.42)
				\psdots[linecolor=black, dotstyle=o, dotsize=0.4, fillcolor=white](8.52,-6.62)
				\psdots[linecolor=black, dotstyle=o, dotsize=0.4, fillcolor=white](9.32,-7.42)
				\psline[linecolor=black, linewidth=0.08](8.92,-6.22)(8.92,-7.02)(8.92,-7.02)
				\psline[linecolor=black, linewidth=0.08](9.72,-7.02)(9.72,-7.82)(9.72,-7.82)
				\psdots[linecolor=black, dotstyle=o, dotsize=0.4, fillcolor=white](1.32,-7.42)
				\psdots[linecolor=black, dotsize=0.4](1.32,-6.62)
				\rput[bl](0.34,1.32){1}
				\rput[bl](0.3,0.5){2}
				\rput[bl](0.28,-0.3){3}
				\rput[bl](0.26,-1.12){4}
				\rput[bl](0.28,-1.94){5}
				\rput[bl](0.28,-2.76){6}
				\rput[bl](0.22,-3.58){7}
				\rput[bl](0.22,-4.32){8}
				\rput[bl](0.02,-6.76){2i-3}
				\rput[bl](0.0,-7.56){2i-2}
				\psline[linecolor=black, linewidth=0.06](1.32,2.98)(5.72,2.98)(5.72,2.98)
				\psline[linecolor=black, linewidth=0.06](11.32,2.98)(12.52,2.98)(12.52,2.98)
				\psdots[linecolor=black, dotsize=0.4](9.32,2.98)
				\psdots[linecolor=black, dotsize=0.4](8.52,2.98)
				\psdots[linecolor=black, dotsize=0.4](7.72,2.98)
				\psline[linecolor=black, linewidth=0.06](7.32,2.98)(9.72,2.98)(9.72,2.98)
				\psdots[linecolor=black, dotsize=0.1](6.12,2.98)
				\psdots[linecolor=black, dotsize=0.1](6.52,2.98)
				\psdots[linecolor=black, dotsize=0.1](6.92,2.98)
				\psdots[linecolor=black, dotsize=0.1](10.12,2.98)
				\psdots[linecolor=black, dotsize=0.1](10.52,2.98)
				\psdots[linecolor=black, dotsize=0.1](10.92,2.98)
				\rput[bl](8.36,3.32){${v_i}$}
				\rput[bl](7.44,3.32){${v_{i-1}}$}
				\rput[bl](9.04,3.3){${v_{i+1}}$}
				\end{pspicture}
			}
		\end{center}
		\caption{\small Making isolate dominating sets of $P_n$ related to the proof of Theorem \ref{path-thm}} \label{path}
	\end{figure}
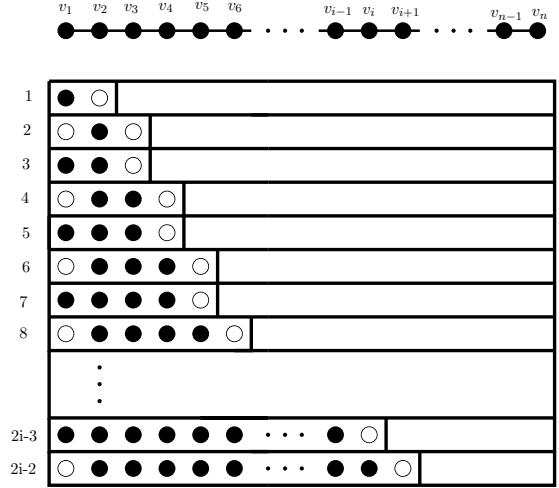

	\begin{theorem}\label{path-thm}
	The number of isolate dominating sets of path $P_n$ with cardinality $i$, $d_0(P_n,i)$, is: 
	\begin{align*}
	d_0(P_n,i) &= d(P_{n-2},i-1) + d(P_{n-3},i-1) \\
	&\quad+ \displaystyle\sum_{k=3}^{i}d_0(P_{n-k},i-k+1)+ \displaystyle\sum_{k=4}^{i+1}d_0(P_{n-k},i-k+2).
	\end{align*}
	\end{theorem}
		
	\begin{proof}
	Let $V(P_n)=\{v_1,v_2,....,v_n\}$ (see Figure \ref{path}) and  $S$ be an isolate  dominating  set of $P_n$ with cardinality $i$. If a graph $G$ contains a simple path of length $3k-1$, then every dominating set of $G$ must contain at least $k$ vertices of the path. So since $S$ is a dominating set of $P_n$,  at least one of the three vertices $v_j,v_{j+1},v_{j+2}$ ($1\leq  j \leq n-2$) are in $S$. So we consider the following cases: 
	\begin{enumerate}
	\item[(1)]
	If $v_1\in S$ and $v_2\not\in S$, then ${\cal D}_0(P_n,i)={\cal D}(P_{n-2},i-1)\cup \{v_1\}$. In this case the number of isolate dominating sets of $P_n$ with cardinality $i$ is $d(P_{n-2},i-1)$.  
	
		\item[(2)]
	If $v_1\not\in S, v_2\in S$ and $v_3\not\in S$, then ${\cal D}_0(P_n,i)={\cal D}(P_{n-3},i-1)\cup \{v_2\}$. In this case the number of isolate dominating sets of $P_n$ with cardinality $i$ is $d(P_{n-3},i-1)$.  
	
		\item[(3)]
			If $v_1,v_2\in S,$ and $v_3\not\in S$, then ${\cal D}_0(P_n,i)={\cal D}_0(P_{n-3},i-2)\cup \{v_1,v_2\}$. In this case the number of isolate dominating sets of $P_n$ with cardinality $i$ is $d_0(P_{n-3},i-2)$. 
		
			\item[(4)]
			If $v_2,v_3\in S,$ and $v_1,v_4\not\in S$, then ${\cal D}_0(P_n,i)={\cal D}(P_{n-3},i-2)\cup \{v_2,v_3\}$. In this case the number of isolate dominating sets of $P_n$ with cardinality $i$ is $d(P_{n-4},i-2)$. 
			
			\item[(5)]
				If $v_1,v_2,v_3\in S$ and $v_4\not\in S$, then ${\cal D}_0(P_n,i)={\cal D}_0(P_{n-4},i-2)\cup \{v_1,v_2,v_3\}$. In this case the number of isolate dominating sets of $P_n$ with cardinality $i$ is $d_0(P_{n-4},i-3)$. 
		\end{enumerate}
		Also we have the following cases:
		
		\begin{enumerate}
			\item[(6)]
			If $v_2,v_3,v_4\in S$ and $v_1,v_5\not\in S$, then ${\cal D}_0(P_n,i)={\cal D}_0(P_{n-4},i-3)\cup \{v_2,v_3,v_4\}$. In this case the number of isolate dominating sets of $P_n$ with cardinality $i$ is $d_0(P_{n-5},i-3)$. 	
			\item[(7)]
				If $v_1,v_2,v_3,v_4\in S$ and $v_5\not\in S$, then ${\cal D}_0(P_n,i)={\cal D}_0(P_{n-4},i-2)\cup \{v_1,v_2,v_3,v_4\}$. In this case the number of isolate dominating sets of $P_n$ with cardinality $i$ is $d_0(P_{n-5},i-4)$. 			
			\item[(8)]
			If $v_2,v_3,v_4,v_5\in S$ and $v_1,v_6\not\in S$, then ${\cal D}_0(P_n,i)={\cal D}_0(P_{n-6},i-4)\cup \{v_2,v_3,v_4,v_5\}$. In this case the number of isolate dominating sets of $P_n$ with cardinality $i$ is $d_0(P_{n-6},i-4)$. 	
			
			\medskip
			By continuing these steps we will have the following two end steps: 
			
			\item[(2i-3)]
			If $v_1,v_2,v_3,\ldots,v_{i-1}\in S$ and $v_i\not\in S$, then
			 $${\cal D}_0(P_n,i)={\cal D}_0(P_{n-i},1)\cup \{v_1,v_2,v_3,\ldots,v_{i-1}\}.$$
			 In this case the number of isolate dominating sets of $P_n$ with cardinality $i$ is $d_0(P_{n-i},1)$. 	
			\item[(2i-2)]
			If $v_2,v_3,v_4,\ldots,v_i\in S,$ and $v_1,v_{i+1}\not\in S$, then 
			$${\cal D}_0(P_n,i)={\cal D}_0(P_{n-i-1},1)\cup \{v_2,v_3,v_4,\ldots,v_i\}.$$
			 In this case the number of isolate dominating sets of $P_n$ with cardinality $i$ is $d_0(P_{n-i-1},1)$. 
	\end{enumerate} 
			So we have:
			\begin{align*}
	d_0(P_n,i) & =  d(P_{n-2},i-1) + d(P_{n-3},i-1) + d_0(P_{n-3},i-2) 
	 + d_0(P_{n-4},i-2) \\
	 &\quad + d_0(P_{n-4},i-3)+d_0(P_{n-5},i-3)
	+ d_0(P_{n-5},i-4)+d_0(P_{n-6},i-4)\\
	&\quad  + \ldots  +d_0(P_{n-i},1)+d_0(P_{n-i-1},1)\\
	&= d(P_{n-2},i-1) + d(P_{n-3},i-1) \\
	&\quad+ \displaystyle\sum_{k=3}^{i}d_0(P_{n-k},i-k+1)+ \displaystyle\sum_{k=4}^{i+1}d_0(P_{n-k},i-k+2)
	\end{align*}
	and therefore we have the result.\qed	
	\end{proof}

\[
\begin{footnotesize}
\small{
	\begin{tabular}{r|lcrrrcccccccc}
	$j$&$1$&$2$&$3$&$4$&$5$&$6$&$7$&$8$&$9$&$10$&$11$&$12$\\[0.3ex]
	\hline
	$n$&&&&&&&&&&&&\\
	$1$&1&&&&&&&&&&&\\
	$2$&2&0&&&&&&&&&&\\
	$3$&1&1&0&&&&&&&&&\\
	$4$&0&3&2&0&&&&&&&&\\
	$5$&0&3&7&2&0&&&&&&&\\
	$6$&0&1&10&9&2&0&&&&&&\\
	$7$&0&0&8&19&12&2&0&&&&&\\
	$8$&0&0&4&25&34&15&2&0&&&&\\
	$9$&0&0&1&22&59&52&18&2&0&&&\\
	$10$&0&0&0&13&70&111&74&20&2&0&&\\
	$11$&0&0&0&5&61&167&192&100&24&2&0&\\
	$12$&0&0&0&1&40&191&344&297&130&27&2&0\\
	\end{tabular}}
\end{footnotesize}
\]
\begin{center}
	\noindent{Table 1.} $d_0(P_{n},j)$ The number of isolate dominating sets of $P_n$ with cardinality $j$.
\end{center}

Using Theorem \ref{path-thm}, we obtain $d_0(P_n,j)$ for $1\leq n\leq 12$ as shown in Table 1.

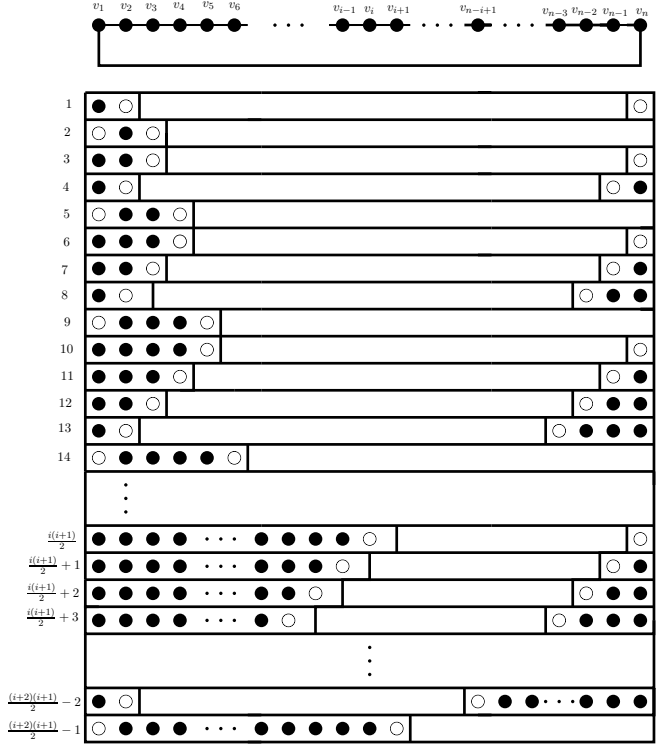
\begin{figure}
	\begin{center}
		\psscalebox{0.450 0.450}
		{
			\begin{pspicture}(0,-12.27)(19.15,9.69)
			\psdots[linecolor=black, dotsize=0.4](2.7,8.98)
			\psdots[linecolor=black, dotsize=0.4](3.5,8.98)
			\psdots[linecolor=black, dotsize=0.4](4.3,8.98)
			\psdots[linecolor=black, dotsize=0.4](5.1,8.98)
			\psdots[linecolor=black, dotsize=0.4](5.9,8.98)
			\psdots[linecolor=black, dotsize=0.4](6.7,8.98)
			\psline[linecolor=black, linewidth=0.08](2.3,6.98)(7.5,6.98)(7.5,6.98)
			\psline[linecolor=black, linewidth=0.08](2.3,6.18)(2.3,6.18)(7.5,6.18)(7.5,6.18)
			\psline[linecolor=black, linewidth=0.08](3.9,6.98)(3.9,6.18)(3.9,6.18)
			\psline[linecolor=black, linewidth=0.08](2.3,5.38)(7.5,5.38)(7.5,5.38)
			\psline[linecolor=black, linewidth=0.08](2.3,4.58)(7.5,4.58)(7.5,4.58)
			\psdots[linecolor=black, dotsize=0.4](2.7,6.58)
			\psdots[linecolor=black, dotsize=0.4](3.5,5.78)
			\psdots[linecolor=black, dotsize=0.4](3.5,4.98)
			\psdots[linecolor=black, dotsize=0.4](2.7,4.98)
			\psdots[linecolor=black, fillstyle=solid, dotstyle=o, dotsize=0.4, fillcolor=white](3.5,6.58)
			\psdots[linecolor=black, fillstyle=solid, dotstyle=o, dotsize=0.4, fillcolor=white](4.3,5.78)
			\psdots[linecolor=black, fillstyle=solid, dotstyle=o, dotsize=0.4, fillcolor=white](2.7,5.78)
			\psdots[linecolor=black, fillstyle=solid, dotstyle=o, dotsize=0.4, fillcolor=white](4.3,4.98)
			\psdots[linecolor=black, fillstyle=solid, dotstyle=o, dotsize=0.4, fillcolor=white](5.1,3.38)
			\psline[linecolor=black, linewidth=0.08](2.3,1.38)(7.5,1.38)(7.5,1.38)
			\psline[linecolor=black, linewidth=0.08](2.3,0.58)(7.5,0.58)(7.5,0.58)
			\psdots[linecolor=black, fillstyle=solid, dotstyle=o, dotsize=0.4, fillcolor=white](2.7,3.38)
			\psdots[linecolor=black, dotsize=0.4](3.5,3.38)
			\psdots[linecolor=black, dotsize=0.4](4.3,3.38)
			\psdots[linecolor=black, dotsize=0.4](2.7,2.58)
			\psdots[linecolor=black, dotsize=0.4](3.5,2.58)
			\psdots[linecolor=black, dotsize=0.4](4.3,2.58)
			\psdots[linecolor=black, dotsize=0.4](5.1,2.58)
			\psdots[linecolor=black, fillstyle=solid, dotstyle=o, dotsize=0.4, fillcolor=white](5.9,0.18)
			\psdots[linecolor=black, fillstyle=solid, dotstyle=o, dotsize=0.4, fillcolor=white](2.7,0.18)
			\psline[linecolor=black, linewidth=0.08](2.3,-1.02)(7.5,-1.02)(7.5,-1.02)
			\psline[linecolor=black, linewidth=0.08](2.3,-0.22)(7.5,-0.22)(7.5,-0.22)
			\psdots[linecolor=black, dotsize=0.4](3.5,0.18)
			\psdots[linecolor=black, dotsize=0.4](4.3,0.18)
			\psdots[linecolor=black, dotsize=0.4](5.1,0.18)
			\psdots[linecolor=black, dotsize=0.4](2.7,-0.62)
			\psdots[linecolor=black, dotsize=0.4](3.5,-0.62)
			\psdots[linecolor=black, dotsize=0.4](4.3,-0.62)
			\psdots[linecolor=black, dotsize=0.4](5.1,-0.62)
			\psdots[linecolor=black, dotsize=0.4](3.5,-3.82)
			\psdots[linecolor=black, dotsize=0.4](4.3,-3.82)
			\psdots[linecolor=black, dotsize=0.4](5.1,-3.82)
			\psdots[linecolor=black, dotsize=0.4](5.9,-3.82)
			\psdots[linecolor=black, fillstyle=solid, dotstyle=o, dotsize=0.4, fillcolor=white](5.9,-0.62)
			\psdots[linecolor=black, fillstyle=solid, dotstyle=o, dotsize=0.4, fillcolor=white](6.7,-3.82)
			\psdots[linecolor=black, fillstyle=solid, dotstyle=o, dotsize=0.4, fillcolor=white](2.7,-3.82)
			\psline[linecolor=black, linewidth=0.08](6.7,-1.82)(7.5,-1.82)(7.5,-1.82)
			\psdots[linecolor=black, dotsize=0.1](3.5,-4.62)
			\psdots[linecolor=black, dotsize=0.1](3.5,-5.02)
			\psdots[linecolor=black, dotsize=0.1](3.5,-5.42)
			\psline[linecolor=black, linewidth=0.08](2.3,-10.62)(2.3,-11.42)(7.5,-11.42)(7.5,-11.42)
			\psdots[linecolor=black, dotsize=0.4](3.5,-6.22)
			\psdots[linecolor=black, dotsize=0.4](4.3,-6.22)
			\psdots[linecolor=black, dotsize=0.4](5.1,-6.22)
			\psdots[linecolor=black, dotsize=0.4](3.5,-11.82)
			\psdots[linecolor=black, dotsize=0.4](4.3,-11.82)
			\psdots[linecolor=black, dotsize=0.4](5.1,-11.82)
			\psline[linecolor=black, linewidth=0.08](2.3,-11.42)(2.3,-12.22)(7.5,-12.22)(7.5,-12.22)
			\psdots[linecolor=black, dotsize=0.4](17.9,8.98)
			\psdots[linecolor=black, dotsize=0.4](18.7,8.98)
			\psline[linecolor=black, linewidth=0.08](7.1,-11.42)(14.3,-11.42)(14.3,-11.42)
			\psline[linecolor=black, linewidth=0.08](7.5,-5.82)(14.3,-5.82)(14.3,-5.82)
			\psline[linecolor=black, linewidth=0.08](7.5,-1.82)(14.3,-1.82)(14.3,-1.82)
			\psline[linecolor=black, linewidth=0.08](7.5,-1.02)(14.3,-1.02)(14.3,-1.02)
			\psline[linecolor=black, linewidth=0.08](7.5,-0.22)(14.3,-0.22)(14.3,-0.22)
			\psline[linecolor=black, linewidth=0.08](7.5,0.58)(14.3,0.58)(14.3,0.58)
			\psline[linecolor=black, linewidth=0.08](7.5,1.38)(7.5,1.38)(7.5,1.38)(7.5,1.38)(14.3,1.38)(14.3,1.38)
			\psline[linecolor=black, linewidth=0.08](7.5,4.58)(14.3,4.58)(14.3,4.58)
			\psline[linecolor=black, linewidth=0.08](7.5,5.38)(14.3,5.38)(14.3,5.38)
			\psline[linecolor=black, linewidth=0.08](7.1,6.18)(14.3,6.18)(14.3,6.18)
			\rput[bl](2.52,9.38){${v_1}$}
			\rput[bl](3.34,9.4){${v_2}$}
			\rput[bl](4.08,9.4){${v_3}$}
			\rput[bl](4.9,9.42){${v_4}$}
			\rput[bl](5.76,9.44){${v_5}$}
			\rput[bl](6.54,9.42){${v_6}$}
			\rput[bl](18.56,9.26){${v_n}$}
			\rput[bl](17.58,9.22){${v_{n-1}}$}
			\psdots[linecolor=black, dotsize=0.4](9.9,-6.22)
			\psdots[linecolor=black, fillstyle=solid, dotstyle=o, dotsize=0.4, fillcolor=white](10.7,-6.22)
			\psline[linecolor=black, linewidth=0.08](11.5,-5.82)(11.5,-6.62)(11.5,-6.62)
			\psline[linecolor=black, linewidth=0.08](11.9,-11.42)(11.9,-12.22)(11.9,-12.22)
			\psdots[linecolor=black, fillstyle=solid, dotstyle=o, dotsize=0.4, fillcolor=white](2.7,-11.82)
			\psdots[linecolor=black, dotsize=0.4](2.7,-6.22)
			\rput[bl](1.72,6.52){1}
			\rput[bl](1.68,5.7){2}
			\rput[bl](1.66,4.9){3}
			\rput[bl](1.64,4.08){4}
			\rput[bl](1.66,3.26){5}
			\rput[bl](1.66,2.44){6}
			\rput[bl](1.6,1.62){7}
			\rput[bl](1.6,0.88){8}
			\psline[linecolor=black, linewidth=0.06](2.7,8.98)(7.1,8.98)(7.1,8.98)
			\psline[linecolor=black, linewidth=0.06](17.5,8.98)(18.7,8.98)(18.7,8.98)
			\psdots[linecolor=black, dotsize=0.4](11.5,8.98)
			\psdots[linecolor=black, dotsize=0.4](10.7,8.98)
			\psdots[linecolor=black, dotsize=0.4](9.9,8.98)
			\psline[linecolor=black, linewidth=0.06](9.5,8.98)(11.9,8.98)(11.9,8.98)
			\rput[bl](10.54,9.32){${v_i}$}
			\rput[bl](9.62,9.32){${v_{i-1}}$}
			\rput[bl](11.22,9.3){${v_{i+1}}$}
			\psline[linecolor=black, linewidth=0.08](6.7,-1.82)(2.3,-1.82)(2.3,-1.82)
			\psline[linecolor=black, linewidth=0.08](2.3,2.18)(14.3,2.18)(14.3,2.18)
			\psline[linecolor=black, linewidth=0.08](2.3,2.98)(14.3,2.98)(14.3,2.98)
			\psline[linecolor=black, linewidth=0.08](2.3,3.78)(14.3,3.78)(14.3,3.78)
			\psline[linecolor=black, linewidth=0.08](7.5,6.98)(14.3,6.98)(13.9,6.98)
			\psdots[linecolor=black, dotsize=0.4](2.7,4.18)
			\psdots[linecolor=black, dotsize=0.4](17.1,8.98)
			\psdots[linecolor=black, dotsize=0.4](16.3,8.98)
			\psline[linecolor=black, linewidth=0.08](15.9,8.98)(17.9,8.98)(17.5,8.98)
			\psline[linecolor=black, linewidth=0.08](2.7,8.98)(2.7,7.78)(18.7,7.78)(18.7,8.98)(18.3,8.98)
			\rput[bl](16.66,9.24){$v_{n-2}$}
			\rput[bl](15.8,9.22){$v_{n-3}$}
			\psline[linecolor=black, linewidth=0.08](14.3,6.98)(19.1,6.98)(19.1,-4.62)(19.1,-4.62)
			\psline[linecolor=black, linewidth=0.08](2.3,-4.22)(2.3,6.98)(2.3,6.98)
			\psline[linecolor=black, linewidth=0.08](14.3,6.18)(19.1,6.18)(19.1,6.18)
			\psline[linecolor=black, linewidth=0.08](13.9,5.38)(19.1,5.38)(19.1,5.38)
			\psline[linecolor=black, linewidth=0.08](13.9,4.58)(19.1,4.58)(19.1,4.58)
			\psline[linecolor=black, linewidth=0.08](14.3,3.78)(19.1,3.78)(19.1,3.78)
			\psline[linecolor=black, linewidth=0.08](14.3,2.98)(19.1,2.98)(18.7,2.98)
			\psline[linecolor=black, linewidth=0.08](13.9,2.18)(19.1,2.18)(19.1,2.18)
			\psline[linecolor=black, linewidth=0.08](14.3,1.38)(19.1,1.38)(19.1,1.38)
			\psline[linecolor=black, linewidth=0.08](14.3,0.58)(19.1,0.58)(18.7,0.58)
			\psline[linecolor=black, linewidth=0.08](14.3,-0.22)(18.7,-0.22)(19.1,-0.22)(19.1,-0.22)
			\psline[linecolor=black, linewidth=0.08](14.3,-1.02)(19.1,-1.02)(19.1,-1.02)
			\psline[linecolor=black, linewidth=0.08](14.3,-1.82)(19.1,-1.82)(19.1,-1.82)
			\psline[linecolor=black, linewidth=0.08](4.7,6.18)(4.7,5.38)(4.7,5.78)
			\psline[linecolor=black, linewidth=0.08](18.3,6.98)(18.3,6.18)(18.3,6.18)
			\psdots[linecolor=black, fillstyle=solid, dotstyle=o, dotsize=0.4, fillcolor=white](18.7,6.58)
			\psline[linecolor=black, linewidth=0.08](4.7,5.38)(4.7,4.58)(4.7,4.58)
			\psline[linecolor=black, linewidth=0.08](18.3,5.38)(18.3,4.58)(18.3,4.58)
			\psdots[linecolor=black, fillstyle=solid, dotstyle=o, dotsize=0.4, fillcolor=white](18.7,4.98)
			\psdots[linecolor=black, fillstyle=solid, dotstyle=o, dotsize=0.4, fillcolor=white](3.5,4.18)
			\psline[linecolor=black, linewidth=0.08](3.9,4.58)(3.9,3.78)(3.9,3.78)
			\psdots[linecolor=black, dotsize=0.4](18.7,4.18)
			\psline[linecolor=black, linewidth=0.08](5.5,3.78)(5.5,2.98)(5.5,2.98)
			\psline[linecolor=black, linewidth=0.08](5.5,2.98)(5.5,2.18)(5.5,2.18)
			\psline[linecolor=black, linewidth=0.08](18.3,2.98)(18.3,2.18)(18.3,2.18)
			\psdots[linecolor=black, fillstyle=solid, dotstyle=o, dotsize=0.4, fillcolor=white](18.7,2.58)
			\psdots[linecolor=black, dotsize=0.4](2.7,1.78)
			\psdots[linecolor=black, dotsize=0.4](3.5,1.78)
			\psdots[linecolor=black, dotsize=0.4](18.7,1.78)
			\psline[linecolor=black, linewidth=0.08](4.7,2.18)(4.7,1.38)(4.7,1.38)
			\psline[linecolor=black, linewidth=0.08](17.5,2.18)(17.5,1.38)(17.5,1.38)
			\psdots[linecolor=black, fillstyle=solid, dotstyle=o, dotsize=0.4, fillcolor=white](4.3,1.78)
			\psdots[linecolor=black, fillstyle=solid, dotstyle=o, dotsize=0.4, fillcolor=white](17.9,1.78)
			\psdots[linecolor=black, dotsize=0.4](2.7,0.98)
			\psdots[linecolor=black, dotsize=0.4](18.7,0.98)
			\psdots[linecolor=black, dotsize=0.4](17.9,0.98)
			\psdots[linecolor=black, fillstyle=solid, dotstyle=o, dotsize=0.4, fillcolor=white](17.1,0.98)
			\psdots[linecolor=black, fillstyle=solid, dotstyle=o, dotsize=0.4, fillcolor=white](3.5,0.98)
			\psline[linecolor=black, linewidth=0.08](4.3,1.38)(4.3,0.58)(4.3,0.58)
			\psline[linecolor=black, linewidth=0.08](16.7,1.38)(16.7,0.58)(16.7,0.58)
			\psline[linecolor=black, linewidth=0.08](6.3,0.58)(6.3,-0.22)(6.3,-0.22)
			\psline[linecolor=black, linewidth=0.08](6.3,-0.22)(6.3,-1.02)(6.3,-1.02)
			\psline[linecolor=black, linewidth=0.08](2.3,-2.62)(19.1,-2.62)(19.1,-2.62)
			\psline[linecolor=black, linewidth=0.08](2.3,-3.42)(19.1,-3.42)(19.1,-3.42)
			\psline[linecolor=black, linewidth=0.08](2.3,-4.22)(19.1,-4.22)(19.1,-4.22)
			\psline[linecolor=black, linewidth=0.08](2.3,-4.22)(2.3,-7.42)(2.3,-7.42)
			\psline[linecolor=black, linewidth=0.08](19.1,-4.22)(19.1,-8.62)(19.1,-8.62)
			\psdots[linecolor=black, fillstyle=solid, dotstyle=o, dotsize=0.4, fillcolor=white](5.1,2.58)
			\psdots[linecolor=black, dotsize=0.4](2.7,-1.42)
			\psdots[linecolor=black, dotsize=0.4](3.5,-1.42)
			\psdots[linecolor=black, dotsize=0.4](4.3,-1.42)
			\psdots[linecolor=black, dotsize=0.4](2.7,-2.22)
			\psdots[linecolor=black, dotsize=0.4](3.5,-2.22)
			\psdots[linecolor=black, dotsize=0.4](2.7,-3.02)
			\psdots[linecolor=black, dotsize=0.4](18.7,-1.42)
			\psdots[linecolor=black, dotsize=0.4](18.7,-2.22)
			\psdots[linecolor=black, dotsize=0.4](17.9,-2.22)
			\psdots[linecolor=black, dotsize=0.4](18.7,-3.02)
			\psdots[linecolor=black, dotsize=0.4](17.9,-3.02)
			\psdots[linecolor=black, dotsize=0.4](17.1,-3.02)
			\psdots[linecolor=black, fillstyle=solid, dotstyle=o, dotsize=0.4, fillcolor=white](18.7,-0.62)
			\psdots[linecolor=black, fillstyle=solid, dotstyle=o, dotsize=0.4, fillcolor=white](17.9,-1.42)
			\psdots[linecolor=black, fillstyle=solid, dotstyle=o, dotsize=0.4, fillcolor=white](17.1,-2.22)
			\psdots[linecolor=black, fillstyle=solid, dotstyle=o, dotsize=0.4, fillcolor=white](16.3,-3.02)
			\psline[linecolor=black, linewidth=0.08](18.3,-0.22)(18.3,-1.02)(18.3,-1.02)
			\psline[linecolor=black, linewidth=0.08](17.5,-1.02)(17.5,-1.82)(17.5,-1.82)
			\psline[linecolor=black, linewidth=0.08](16.7,-1.82)(16.7,-2.62)(16.7,-2.62)
			\psline[linecolor=black, linewidth=0.08](15.9,-2.62)(15.9,-3.42)(15.9,-3.42)
			\psline[linecolor=black, linewidth=0.08](5.5,-1.02)(5.5,-1.82)(5.1,-1.82)
			\psline[linecolor=black, linewidth=0.08](4.7,-1.82)(4.7,-2.62)(4.7,-2.62)
			\psline[linecolor=black, linewidth=0.08](3.9,-2.62)(3.9,-3.42)(3.9,-3.42)
			\psdots[linecolor=black, fillstyle=solid, dotstyle=o, dotsize=0.4, fillcolor=white](5.1,-1.42)
			\psdots[linecolor=black, fillstyle=solid, dotstyle=o, dotsize=0.4, fillcolor=white](4.3,-2.22)
			\psdots[linecolor=black, fillstyle=solid, dotstyle=o, dotsize=0.4, fillcolor=white](3.5,-3.02)
			\psline[linecolor=black, linewidth=0.08](2.3,-6.62)(19.1,-6.62)(19.1,-6.62)
			\psline[linecolor=black, linewidth=0.08](7.5,-5.82)(2.3,-5.82)(2.3,-5.82)
			\psline[linecolor=black, linewidth=0.08](14.3,-5.82)(19.1,-5.82)(19.1,-5.82)
			\psdots[linecolor=black, fillstyle=solid, dotstyle=o, dotsize=0.4, fillcolor=white](18.7,-6.22)
			\psdots[linecolor=black, fillstyle=solid, dotstyle=o, dotsize=0.4, fillcolor=white](17.9,-7.02)
			\psdots[linecolor=black, fillstyle=solid, dotstyle=o, dotsize=0.4, fillcolor=white](17.1,-7.82)
			\psdots[linecolor=black, fillstyle=solid, dotstyle=o, dotsize=0.4, fillcolor=white](16.3,-8.62)
			\psline[linecolor=black, linewidth=0.08](2.3,-7.42)(19.1,-7.42)(19.1,-7.42)
			\psline[linecolor=black, linewidth=0.08](19.1,-8.22)(2.3,-8.22)(2.3,-7.42)(2.3,-7.42)
			\psline[linecolor=black, linewidth=0.08](19.1,-8.62)(19.1,-9.02)(2.3,-9.02)(2.3,-8.22)(2.7,-8.22)
			\psdots[linecolor=black, dotsize=0.4](3.5,-7.02)
			\psdots[linecolor=black, dotsize=0.4](4.3,-7.02)
			\psdots[linecolor=black, dotsize=0.4](5.1,-7.02)
			\psdots[linecolor=black, dotsize=0.4](2.7,-7.02)
			\psdots[linecolor=black, dotsize=0.4](3.5,-7.82)
			\psdots[linecolor=black, dotsize=0.4](4.3,-7.82)
			\psdots[linecolor=black, dotsize=0.4](5.1,-7.82)
			\psdots[linecolor=black, dotsize=0.4](2.7,-7.82)
			\psdots[linecolor=black, dotsize=0.4](3.5,-8.62)
			\psdots[linecolor=black, dotsize=0.4](4.3,-8.62)
			\psdots[linecolor=black, dotsize=0.4](5.1,-8.62)
			\psdots[linecolor=black, dotsize=0.4](2.7,-8.62)
			\psdots[linecolor=black, dotsize=0.4](9.1,-6.22)
			\psdots[linecolor=black, dotsize=0.4](8.3,-6.22)
			\psdots[linecolor=black, fillstyle=solid, dotstyle=o, dotsize=0.4, fillcolor=white](9.9,-7.02)
			\psdots[linecolor=black, fillstyle=solid, dotstyle=o, dotsize=0.4, fillcolor=white](9.1,-7.82)
			\psdots[linecolor=black, fillstyle=solid, dotstyle=o, dotsize=0.4, fillcolor=white](8.3,-8.62)
			\psdots[linecolor=black, dotsize=0.1](10.7,-9.42)
			\psdots[linecolor=black, dotsize=0.1](10.7,-9.82)
			\psdots[linecolor=black, dotsize=0.1](10.7,-10.22)
			\psline[linecolor=black, linewidth=0.08](2.3,-9.02)(2.3,-10.62)(19.1,-10.62)(19.1,-8.62)(19.1,-8.62)
			\psline[linecolor=black, linewidth=0.08](7.1,-12.22)(19.1,-12.22)(19.1,-10.62)(19.1,-11.42)(14.3,-11.42)(13.9,-11.42)
			\psdots[linecolor=black, dotsize=0.4](2.7,-11.02)
			\psdots[linecolor=black, dotsize=0.4](18.7,-11.02)
			\psdots[linecolor=black, dotsize=0.4](17.9,-11.02)
			\psdots[linecolor=black, dotsize=0.1](15.9,-11.02)
			\psdots[linecolor=black, dotsize=0.4](15.5,-11.02)
			\psdots[linecolor=black, dotsize=0.1](16.3,-11.02)
			\psdots[linecolor=black, dotsize=0.4](14.7,-11.02)
			\psdots[linecolor=black, fillstyle=solid, dotstyle=o, dotsize=0.4, fillcolor=white](13.9,-11.02)
			\psline[linecolor=black, linewidth=0.08](15.9,-8.22)(15.9,-9.02)(15.9,-9.02)
			\psline[linecolor=black, linewidth=0.08](10.7,-6.62)(10.7,-7.42)(9.9,-7.42)(9.9,-8.22)(9.1,-8.22)(9.1,-9.02)(9.1,-9.02)
			\psline[linecolor=black, linewidth=0.08](18.3,-5.82)(18.3,-6.62)(17.9,-6.62)(17.5,-6.62)(17.5,-7.42)(16.7,-7.42)(16.7,-8.22)(16.7,-8.22)
			\psline[linecolor=black, linewidth=0.08](13.5,-10.62)(13.5,-11.42)(13.5,-11.42)
			\psline[linecolor=black, linewidth=0.08](3.9,-11.42)(3.9,-10.62)(3.9,-10.62)
			\psdots[linecolor=black, fillstyle=solid, dotstyle=o, dotsize=0.4, fillcolor=white](3.5,-11.02)
			\psdots[linecolor=black, dotsize=0.4](18.7,-7.02)
			\psdots[linecolor=black, dotsize=0.4](17.9,-7.82)
			\psdots[linecolor=black, dotsize=0.4](18.7,-7.82)
			\psdots[linecolor=black, dotsize=0.4](17.1,-8.62)
			\psdots[linecolor=black, dotsize=0.4](17.9,-8.62)
			\psdots[linecolor=black, dotsize=0.4](18.7,-8.62)
			\rput[bl](1.68,0.08){9}
			\rput[bl](1.56,-0.72){10}
			\rput[bl](1.58,-1.52){11}
			\rput[bl](1.52,-2.3){12}
			\rput[bl](1.5,-3.06){13}
			\rput[bl](1.48,-3.9){14}
			\rput[bl](0.0,-12.12){$\frac{(i+2)(i+1)}{2}-1$}
			\rput[bl](0.0,-11.3){$\frac{(i+2)(i+1)}{2}-2$}
			\rput[bl](1.16,-6.52){$\frac{i(i+1)}{2}$}
			\rput[bl](0.6,-7.24){$\frac{i(i+1)}{2}+1$}
			\rput[bl](0.56,-8.06){$\frac{i(i+1)}{2}+2$}
			\rput[bl](0.54,-8.78){$\frac{i(i+1)}{2}+3$}
			\psline[linecolor=black, linewidth=0.08](7.1,-3.42)(7.1,-4.22)(7.1,-4.22)
			\psdots[linecolor=black, dotsize=0.4](9.1,-7.02)
			\psdots[linecolor=black, dotsize=0.4](8.3,-7.82)
			\psdots[linecolor=black, dotsize=0.4](8.3,-7.02)
			\psdots[linecolor=black, dotsize=0.4](7.5,-8.62)
			\psdots[linecolor=black, dotsize=0.4](7.5,-7.82)
			\psdots[linecolor=black, dotsize=0.4](7.5,-7.02)
			\psdots[linecolor=black, dotsize=0.4](7.5,-6.22)
			\psdots[linecolor=black, dotsize=0.1](5.9,-6.22)
			\psdots[linecolor=black, dotsize=0.1](6.3,-6.22)
			\psdots[linecolor=black, dotsize=0.1](6.7,-6.22)
			\psdots[linecolor=black, dotsize=0.1](5.9,-7.02)
			\psdots[linecolor=black, dotsize=0.1](6.3,-7.02)
			\psdots[linecolor=black, dotsize=0.1](6.7,-7.02)
			\psdots[linecolor=black, dotsize=0.1](5.9,-7.82)
			\psdots[linecolor=black, dotsize=0.1](6.3,-7.82)
			\psdots[linecolor=black, dotsize=0.1](6.7,-7.82)
			\psdots[linecolor=black, dotsize=0.1](5.9,-8.62)
			\psdots[linecolor=black, dotsize=0.1](6.3,-8.62)
			\psdots[linecolor=black, dotsize=0.1](6.7,-8.62)
			\psdots[linecolor=black, dotsize=0.1](5.9,-11.82)
			\psdots[linecolor=black, dotsize=0.1](6.3,-11.82)
			\psdots[linecolor=black, dotsize=0.1](6.7,-11.82)
			\psdots[linecolor=black, dotsize=0.4](7.5,-11.82)
			\psdots[linecolor=black, dotsize=0.4](8.3,-11.82)
			\psdots[linecolor=black, dotsize=0.4](9.1,-11.82)
			\psdots[linecolor=black, dotsize=0.4](9.9,-11.82)
			\psdots[linecolor=black, dotsize=0.4](10.7,-11.82)
			\psdots[linecolor=black, fillstyle=solid, dotstyle=o, dotsize=0.4, fillcolor=white](11.5,-11.82)
			\psdots[linecolor=black, dotsize=0.4](17.1,-11.02)
			\psdots[linecolor=black, dotsize=0.1](16.7,-11.02)
			\psdots[linecolor=black, dotsize=0.1](14.7,8.98)
			\psdots[linecolor=black, dotsize=0.1](15.1,8.98)
			\psdots[linecolor=black, dotsize=0.1](15.5,8.98)
			\psdots[linecolor=black, dotsize=0.1](13.1,8.98)
			\psdots[linecolor=black, dotsize=0.1](12.7,8.98)
			\psdots[linecolor=black, dotsize=0.1](12.3,8.98)
			\psdots[linecolor=black, dotsize=0.1](8.7,8.98)
			\psdots[linecolor=black, dotsize=0.1](8.3,8.98)
			\psdots[linecolor=black, dotsize=0.1](7.9,8.98)
			\psdots[linecolor=black, dotsize=0.4](13.9,8.98)
			\psline[linecolor=black, linewidth=0.08](13.5,8.98)(14.3,8.98)(13.9,8.98)(13.9,8.98)
			\rput[bl](13.36,9.3){$v_{n-i+1}$}
			\psline[linecolor=black, linewidth=0.08](17.5,4.58)(17.5,3.78)(17.5,3.78)
			\psdots[linecolor=black, dotstyle=o, dotsize=0.4, fillcolor=white](17.9,4.18)
			\end{pspicture}
		}
	\end{center}
	\caption{\footnotesize Making isolate dominating sets of $C_n$ such that at least one of the vertices $v_1$ and $v_2$ are in $S$, related to Theorem \ref{cycle-thm}} \label{cycle}
\end{figure}

\vspace{1.5cm}
\begin{figure}
	\begin{center}
		\psscalebox{0.450 0.450}
		{			\begin{pspicture}(0,-12.27)(19.183334,9.69)
			\psdots[linecolor=black, dotsize=0.4](2.7333333,8.98)
			\psdots[linecolor=black, dotsize=0.4](3.5333333,8.98)
			\psdots[linecolor=black, dotsize=0.4](4.3333335,8.98)
			\psdots[linecolor=black, dotsize=0.4](5.133333,8.98)
			\psdots[linecolor=black, dotsize=0.4](5.9333334,8.98)
			\psdots[linecolor=black, dotsize=0.4](6.733333,8.98)
			\psline[linecolor=black, linewidth=0.08](2.3333333,6.98)(7.5333333,6.98)(7.5333333,6.98)
			\psline[linecolor=black, linewidth=0.08](2.3333333,6.18)(2.3333333,6.18)(7.5333333,6.18)(7.5333333,6.18)
			\psline[linecolor=black, linewidth=0.08](2.3333333,5.38)(7.5333333,5.38)(7.5333333,5.38)
			\psline[linecolor=black, linewidth=0.08](2.3333333,4.58)(7.5333333,4.58)(7.5333333,4.58)
			\psdots[linecolor=black, fillstyle=solid, dotstyle=o, dotsize=0.4, fillcolor=white](3.5333333,6.58)
			\psdots[linecolor=black, fillstyle=solid, dotstyle=o, dotsize=0.4, fillcolor=white](2.7333333,5.78)
			\psline[linecolor=black, linewidth=0.08](2.3333333,0.58)(7.5333333,0.58)(7.5333333,0.58)
			\psdots[linecolor=black, fillstyle=solid, dotstyle=o, dotsize=0.4, fillcolor=white](2.7333333,3.38)
			\psdots[linecolor=black, dotsize=0.4](4.3333335,3.38)
			\psdots[linecolor=black, dotsize=0.4](4.3333335,2.58)
			\psdots[linecolor=black, fillstyle=solid, dotstyle=o, dotsize=0.4, fillcolor=white](2.7333333,0.18)
			\psline[linecolor=black, linewidth=0.08](2.3333333,-1.02)(7.5333333,-1.02)(7.5333333,-1.02)
			\psline[linecolor=black, linewidth=0.08](2.3333333,-0.22)(7.5333333,-0.22)(7.5333333,-0.22)
			\psdots[linecolor=black, dotsize=0.4](4.3333335,0.18)
			\psdots[linecolor=black, dotsize=0.4](4.3333335,-0.62)
			\psline[linecolor=black, linewidth=0.08](6.733333,-1.82)(7.5333333,-1.82)(7.5333333,-1.82)
			\psline[linecolor=black, linewidth=0.08](2.3333333,-10.62)(2.3333333,-11.42)(7.5333333,-11.42)(7.5333333,-11.42)
			\psdots[linecolor=black, dotsize=0.4](4.3333335,-11.82)
			\psdots[linecolor=black, dotsize=0.4](5.133333,-11.82)
			\psline[linecolor=black, linewidth=0.08](2.3333333,-11.42)(2.3333333,-12.22)(7.5333333,-12.22)(7.5333333,-12.22)
			\psdots[linecolor=black, dotsize=0.4](17.933332,8.98)
			\psdots[linecolor=black, dotsize=0.4](18.733334,8.98)
			\psline[linecolor=black, linewidth=0.08](7.133333,-11.42)(14.333333,-11.42)(14.333333,-11.42)
			\psline[linecolor=black, linewidth=0.08](7.5333333,-5.82)(14.333333,-5.82)(14.333333,-5.82)
			\psline[linecolor=black, linewidth=0.08](7.5333333,-1.82)(14.333333,-1.82)(14.333333,-1.82)
			\psline[linecolor=black, linewidth=0.08](7.5333333,-1.02)(14.333333,-1.02)(14.333333,-1.02)
			\psline[linecolor=black, linewidth=0.08](7.5333333,-0.22)(14.333333,-0.22)(14.333333,-0.22)
			\psline[linecolor=black, linewidth=0.08](7.5333333,0.58)(14.333333,0.58)(14.333333,0.58)
			\psline[linecolor=black, linewidth=0.08](7.5333333,4.58)(14.333333,4.58)(14.333333,4.58)
			\psline[linecolor=black, linewidth=0.08](7.5333333,5.38)(14.333333,5.38)(14.333333,5.38)
			\psline[linecolor=black, linewidth=0.08](7.133333,6.18)(14.333333,6.18)(14.333333,6.18)
			\rput[bl](2.5533333,9.38){${v_1}$}
			\rput[bl](3.3733332,9.4){${v_2}$}
			\rput[bl](4.113333,9.4){${v_3}$}
			\rput[bl](4.9333334,9.42){${v_4}$}
			\rput[bl](5.7933335,9.44){${v_5}$}
			\rput[bl](6.5733333,9.42){${v_6}$}
			\rput[bl](18.593334,9.26){${v_n}$}
			\rput[bl](17.613333,9.22){${v_{n-1}}$}
			\psdots[linecolor=black, fillstyle=solid, dotstyle=o, dotsize=0.4, fillcolor=white](2.7333333,-11.82)
			\rput[bl](1.7533333,6.52){1}
			\rput[bl](1.7133334,5.7){2}
			\rput[bl](1.6933333,4.9){3}
			\rput[bl](1.6733333,4.08){4}
			\rput[bl](1.6933333,3.26){5}
			\rput[bl](1.6933333,2.44){6}
			\psline[linecolor=black, linewidth=0.06](2.7333333,8.98)(7.133333,8.98)(7.133333,8.98)
			\psline[linecolor=black, linewidth=0.06](17.533333,8.98)(18.733334,8.98)(18.733334,8.98)
			\psdots[linecolor=black, dotsize=0.4](11.533334,8.98)
			\psdots[linecolor=black, dotsize=0.4](10.733334,8.98)
			\psdots[linecolor=black, dotsize=0.4](9.933333,8.98)
			\psline[linecolor=black, linewidth=0.06](9.533334,8.98)(11.933333,8.98)(11.933333,8.98)
			\rput[bl](10.573334,9.32){${v_i}$}
			\rput[bl](9.653334,9.32){${v_{i-1}}$}
			\rput[bl](11.253333,9.3){${v_{i+1}}$}
			\psline[linecolor=black, linewidth=0.08](6.733333,-1.82)(2.3333333,-1.82)(2.3333333,-1.82)
			\psline[linecolor=black, linewidth=0.08](2.3333333,2.18)(14.333333,2.18)(14.333333,2.18)
			\psline[linecolor=black, linewidth=0.08](2.3333333,2.98)(14.333333,2.98)(14.333333,2.98)
			\psline[linecolor=black, linewidth=0.08](2.3333333,3.78)(14.333333,3.78)(14.333333,3.78)
			\psline[linecolor=black, linewidth=0.08](7.5333333,6.98)(14.333333,6.98)(13.933333,6.98)
			\psdots[linecolor=black, dotsize=0.4](17.133333,8.98)
			\psdots[linecolor=black, dotsize=0.4](16.333334,8.98)
			\psline[linecolor=black, linewidth=0.08](15.933333,8.98)(17.933332,8.98)(17.533333,8.98)
			\psline[linecolor=black, linewidth=0.08](2.7333333,8.98)(2.7333333,7.78)(18.733334,7.78)(18.733334,8.98)(18.333334,8.98)
			\rput[bl](16.693333,9.24){$v_{n-2}$}
			\rput[bl](15.833333,9.22){$v_{n-3}$}
			\psline[linecolor=black, linewidth=0.08](14.333333,6.98)(19.133333,6.98)(19.133333,-4.62)(19.133333,-4.62)
			\psline[linecolor=black, linewidth=0.08](2.3333333,-4.22)(2.3333333,6.98)(2.3333333,6.98)
			\psline[linecolor=black, linewidth=0.08](14.333333,6.18)(19.133333,6.18)(19.133333,6.18)
			\psline[linecolor=black, linewidth=0.08](13.933333,5.38)(19.133333,5.38)(19.133333,5.38)
			\psline[linecolor=black, linewidth=0.08](13.933333,4.58)(19.133333,4.58)(19.133333,4.58)
			\psline[linecolor=black, linewidth=0.08](14.333333,3.78)(19.133333,3.78)(19.133333,3.78)
			\psline[linecolor=black, linewidth=0.08](14.333333,2.98)(19.133333,2.98)(18.733334,2.98)
			\psline[linecolor=black, linewidth=0.08](13.933333,2.18)(19.133333,2.18)(19.133333,2.18)
			\psline[linecolor=black, linewidth=0.08](14.333333,0.58)(19.133333,0.58)(18.733334,0.58)
			\psline[linecolor=black, linewidth=0.08](14.333333,-0.22)(18.733334,-0.22)(19.133333,-0.22)(19.133333,-0.22)
			\psline[linecolor=black, linewidth=0.08](14.333333,-1.02)(19.133333,-1.02)(19.133333,-1.02)
			\psline[linecolor=black, linewidth=0.08](14.333333,-1.82)(19.133333,-1.82)(19.133333,-1.82)
			\psdots[linecolor=black, fillstyle=solid, dotstyle=o, dotsize=0.4, fillcolor=white](3.5333333,4.18)
			\psdots[linecolor=black, dotsize=0.4](18.733334,4.18)
			\psline[linecolor=black, linewidth=0.08](2.3333333,-2.62)(19.133333,-2.62)(19.133333,-2.62)
			\psline[linecolor=black, linewidth=0.08](2.3333333,-5.02)(19.133333,-5.02)(19.133333,-5.02)
			\psline[linecolor=black, linewidth=0.08](2.3333333,-4.22)(19.133333,-4.22)(19.133333,-4.22)
			\psline[linecolor=black, linewidth=0.08](2.3333333,-4.22)(2.3333333,-7.42)(2.3333333,-7.42)
			\psline[linecolor=black, linewidth=0.08](19.133333,-4.22)(19.133333,-8.62)(19.133333,-8.62)
			\psdots[linecolor=black, dotsize=0.4](4.3333335,-1.42)
			\psdots[linecolor=black, dotsize=0.4](18.733334,-1.42)
			\psdots[linecolor=black, dotsize=0.4](18.733334,-2.22)
			\psline[linecolor=black, linewidth=0.08](2.3333333,-6.62)(19.133333,-6.62)(19.133333,-6.62)
			\psline[linecolor=black, linewidth=0.08](7.5333333,-5.82)(2.3333333,-5.82)(2.3333333,-5.82)
			\psline[linecolor=black, linewidth=0.08](14.333333,-5.82)(19.133333,-5.82)(19.133333,-5.82)
			\psline[linecolor=black, linewidth=0.08](2.3333333,-7.42)(19.133333,-7.42)(19.133333,-7.42)
			\psline[linecolor=black, linewidth=0.08](7.133333,-12.22)(19.133333,-12.22)(19.133333,-10.62)(19.133333,-11.42)(14.333333,-11.42)(13.933333,-11.42)
			\psdots[linecolor=black, dotsize=0.4](18.733334,-7.02)
			\psdots[linecolor=black, dotsize=0.1](14.733334,8.98)
			\psdots[linecolor=black, dotsize=0.1](15.133333,8.98)
			\psdots[linecolor=black, dotsize=0.1](15.533334,8.98)
			\psdots[linecolor=black, dotsize=0.1](13.133333,8.98)
			\psdots[linecolor=black, dotsize=0.1](12.733334,8.98)
			\psdots[linecolor=black, dotsize=0.1](12.333333,8.98)
			\psdots[linecolor=black, dotsize=0.1](8.333333,8.98)
			\psdots[linecolor=black, dotsize=0.1](7.9333334,8.98)
			\psdots[linecolor=black, dotsize=0.1](7.5333333,8.98)
			\psdots[linecolor=black, dotsize=0.4](13.933333,8.98)
			\psline[linecolor=black, linewidth=0.08](13.533334,8.98)(14.333333,8.98)(13.933333,8.98)(13.933333,8.98)
			\psline[linecolor=black, linewidth=0.08](9.533334,8.98)(11.933333,8.98)(11.933333,8.98)
			\psline[linecolor=black, linewidth=0.08](15.933333,8.98)(18.733334,8.98)(18.733334,8.98)
			\psdots[linecolor=black, dotsize=0.4](4.3333335,6.58)
			\psdots[linecolor=black, dotsize=0.4](4.3333335,5.78)
			\psdots[linecolor=black, dotsize=0.4](4.3333335,4.98)
			\psdots[linecolor=black, dotsize=0.4](4.3333335,4.18)
			\psdots[linecolor=black, dotsize=0.4](18.733334,6.58)
			\psdots[linecolor=black, dotsize=0.4](18.733334,5.78)
			\psdots[linecolor=black, dotsize=0.4](18.733334,4.98)
			\psdots[linecolor=black, dotsize=0.4](18.733334,3.38)
			\psdots[linecolor=black, dotsize=0.4](18.733334,2.58)
			\psdots[linecolor=black, dotsize=0.4](18.733334,0.18)
			\psdots[linecolor=black, dotsize=0.4](18.733334,-0.62)
			\psdots[linecolor=black, fillstyle=solid, dotstyle=o, dotsize=0.4, fillcolor=white](2.7333333,6.58)
			\psdots[linecolor=black, fillstyle=solid, dotstyle=o, dotsize=0.4, fillcolor=white](3.5333333,5.78)
			\psdots[linecolor=black, fillstyle=solid, dotstyle=o, dotsize=0.4, fillcolor=white](2.7333333,4.98)
			\psdots[linecolor=black, fillstyle=solid, dotstyle=o, dotsize=0.4, fillcolor=white](3.5333333,4.98)
			\psdots[linecolor=black, fillstyle=solid, dotstyle=o, dotsize=0.4, fillcolor=white](2.7333333,4.18)
			\psdots[linecolor=black, fillstyle=solid, dotstyle=o, dotsize=0.4, fillcolor=white](3.5333333,3.38)
			\psdots[linecolor=black, fillstyle=solid, dotstyle=o, dotsize=0.4, fillcolor=white](2.7333333,2.58)
			\psdots[linecolor=black, fillstyle=solid, dotstyle=o, dotsize=0.4, fillcolor=white](3.5333333,2.58)
			\psdots[linecolor=black, fillstyle=solid, dotstyle=o, dotsize=0.4, fillcolor=white](3.5333333,0.18)
			\psdots[linecolor=black, fillstyle=solid, dotstyle=o, dotsize=0.4, fillcolor=white](2.7333333,-0.62)
			\psdots[linecolor=black, fillstyle=solid, dotstyle=o, dotsize=0.4, fillcolor=white](3.5333333,-0.62)
			\psdots[linecolor=black, fillstyle=solid, dotstyle=o, dotsize=0.4, fillcolor=white](2.7333333,-1.42)
			\psdots[linecolor=black, fillstyle=solid, dotstyle=o, dotsize=0.4, fillcolor=white](3.5333333,-1.42)
			\psdots[linecolor=black, fillstyle=solid, dotstyle=o, dotsize=0.4, fillcolor=white](2.7333333,-2.22)
			\psdots[linecolor=black, fillstyle=solid, dotstyle=o, dotsize=0.4, fillcolor=white](3.5333333,-2.22)
			\psdots[linecolor=black, dotsize=0.4](4.3333335,-2.22)
			\psdots[linecolor=black, fillstyle=solid, dotstyle=o, dotsize=0.4, fillcolor=white](5.133333,6.58)
			\psdots[linecolor=black, fillstyle=solid, dotstyle=o, dotsize=0.4, fillcolor=white](17.933332,6.58)
			\psdots[linecolor=black, fillstyle=solid, dotstyle=o, dotsize=0.4, fillcolor=white](5.133333,5.78)
			\psdots[linecolor=black, fillstyle=solid, dotstyle=o, dotsize=0.4, fillcolor=white](17.933332,4.98)
			\psdots[linecolor=black, dotsize=0.4](17.933332,5.78)
			\psdots[linecolor=black, dotsize=0.4](5.133333,4.98)
			\psdots[linecolor=black, dotsize=0.4](17.933332,4.18)
			\psdots[linecolor=black, dotsize=0.4](5.133333,4.18)
			\psline[linecolor=black, linewidth=0.08](5.5333333,6.98)(5.5333333,6.18)(17.533333,6.18)(17.533333,6.18)
			\psline[linecolor=black, linewidth=0.08](17.533333,6.98)(17.533333,6.18)(17.533333,6.18)
			\psline[linecolor=black, linewidth=0.08](5.5333333,6.18)(5.5333333,5.38)(5.5333333,5.38)(5.5333333,5.38)
			\psline[linecolor=black, linewidth=0.08](17.533333,6.18)(17.533333,5.38)(17.533333,5.38)
			\psline[linecolor=black, linewidth=0.08](5.5333333,5.38)(5.5333333,4.58)(5.5333333,4.58)
			\psline[linecolor=black, linewidth=0.08](17.533333,5.38)(17.533333,4.58)(17.533333,4.58)
			\psdots[linecolor=black, fillstyle=solid, dotstyle=o, dotsize=0.4, fillcolor=white](5.9333334,4.18)
			\psdots[linecolor=black, fillstyle=solid, dotstyle=o, dotsize=0.4, fillcolor=white](17.133333,4.18)
			\psdots[linecolor=black, fillstyle=solid, dotstyle=o, dotsize=0.4, fillcolor=white](17.133333,3.38)
			\psdots[linecolor=black, fillstyle=solid, dotstyle=o, dotsize=0.4, fillcolor=white](17.133333,2.58)
			\psdots[linecolor=black, dotsize=0.4](17.933332,3.38)
			\psdots[linecolor=black, dotsize=0.4](17.933332,2.58)
			\psdots[linecolor=black, dotsize=0.4](5.133333,3.38)
			\psdots[linecolor=black, dotsize=0.4](5.9333334,3.38)
			\psdots[linecolor=black, dotsize=0.4](5.133333,2.58)
			\psdots[linecolor=black, dotsize=0.4](5.9333334,2.58)
			\psdots[linecolor=black, dotsize=0.4](6.733333,2.58)
			\psdots[linecolor=black, dotsize=0.1](10.733334,1.78)
			\psdots[linecolor=black, dotsize=0.1](10.733334,1.38)
			\psdots[linecolor=black, dotsize=0.1](10.733334,0.98)
			\psdots[linecolor=black, dotsize=0.4](17.933332,0.18)
			\psdots[linecolor=black, dotsize=0.4](5.133333,0.18)
			\psline[linecolor=black, linewidth=0.08](2.7333333,8.98)(7.133333,8.98)(7.133333,8.98)
			\psline[linecolor=black, linewidth=0.08](8.733334,8.98)(9.533334,8.98)(9.533334,8.98)
			\psdots[linecolor=black, dotsize=0.4](9.133333,8.98)
			\psdots[linecolor=black, fillstyle=solid, dotstyle=o, dotsize=0.4, fillcolor=white](6.733333,3.38)
			\psdots[linecolor=black, fillstyle=solid, dotstyle=o, dotsize=0.4, fillcolor=white](7.5333333,2.58)
			\psdots[linecolor=black, fillstyle=solid, dotstyle=o, dotsize=0.4, fillcolor=white](17.133333,0.18)
			\psline[linecolor=black, linewidth=0.08](16.733334,4.58)(16.733334,2.18)(16.733334,2.18)
			\psline[linecolor=black, linewidth=0.08](16.733334,0.58)(16.733334,-0.22)(16.733334,-0.22)
			\psline[linecolor=black, linewidth=0.08](11.133333,0.58)(11.133333,-0.22)(11.133333,-0.22)
			\psline[linecolor=black, linewidth=0.08](7.9333334,2.18)(7.9333334,2.98)(7.133333,2.98)(7.133333,3.78)(6.3333335,3.78)(6.3333335,4.58)(6.3333335,4.58)
			\psdots[linecolor=black, fillstyle=solid, dotstyle=o, dotsize=0.4, fillcolor=white](10.733334,0.18)
			\psdots[linecolor=black, dotsize=0.4](9.933333,0.18)
			\psdots[linecolor=black, dotsize=0.4](5.9333334,0.18)
			\psdots[linecolor=black, dotsize=0.1](7.5333333,0.18)
			\psdots[linecolor=black, dotsize=0.1](7.133333,0.18)
			\psdots[linecolor=black, dotsize=0.1](7.9333334,0.18)
			\psdots[linecolor=black, dotsize=0.4](17.933332,-0.62)
			\psdots[linecolor=black, dotsize=0.4](17.133333,-0.62)
			\psdots[linecolor=black, dotsize=0.4](17.933332,-1.42)
			\psdots[linecolor=black, dotsize=0.4](17.133333,-1.42)
			\psdots[linecolor=black, dotsize=0.4](17.933332,-2.22)
			\psdots[linecolor=black, dotsize=0.4](17.133333,-2.22)
			\psdots[linecolor=black, fillstyle=solid, dotstyle=o, dotsize=0.4, fillcolor=white](16.333334,-0.62)
			\psdots[linecolor=black, fillstyle=solid, dotstyle=o, dotsize=0.4, fillcolor=white](16.333334,-1.42)
			\psdots[linecolor=black, fillstyle=solid, dotstyle=o, dotsize=0.4, fillcolor=white](16.333334,-2.22)
			\psline[linecolor=black, linewidth=0.08](15.933333,-0.22)(15.933333,-2.62)(15.933333,-2.62)
			\psdots[linecolor=black, dotsize=0.1](10.733334,-3.02)
			\psdots[linecolor=black, dotsize=0.1](10.733334,-3.42)
			\psdots[linecolor=black, dotsize=0.1](10.733334,-3.82)
			\psdots[linecolor=black, dotsize=0.4](9.133333,-4.62)
			\psdots[linecolor=black, dotsize=0.4](17.133333,-4.62)
			\psdots[linecolor=black, dotsize=0.4](17.933332,-4.62)
			\psdots[linecolor=black, dotsize=0.4](18.733334,-4.62)
			\psdots[linecolor=black, fillstyle=solid, dotstyle=o, dotsize=0.4, fillcolor=white](16.333334,-4.62)
			\psline[linecolor=black, linewidth=0.08](15.933333,-4.22)(15.933333,-5.02)(15.933333,-5.02)
			\psline[linecolor=black, linewidth=0.08](10.333333,-4.22)(10.333333,-5.02)(10.333333,-5.02)
			\psline[linecolor=black, linewidth=0.08](6.3333335,-0.22)(6.3333335,-1.02)(7.133333,-1.02)(7.133333,-1.82)(7.9333334,-1.82)(7.9333334,-2.62)(7.9333334,-2.62)
			\psdots[linecolor=black, dotsize=0.4](5.133333,-0.62)
			\psdots[linecolor=black, dotsize=0.4](5.133333,-1.42)
			\psdots[linecolor=black, dotsize=0.4](5.9333334,-1.42)
			\psdots[linecolor=black, dotsize=0.4](5.133333,-2.22)
			\psdots[linecolor=black, dotsize=0.4](5.9333334,-2.22)
			\psdots[linecolor=black, dotsize=0.4](6.733333,-2.22)
			\psdots[linecolor=black, dotsize=0.4](4.3333335,-4.62)
			\psdots[linecolor=black, dotsize=0.4](5.133333,-4.62)
			\psdots[linecolor=black, fillstyle=solid, dotstyle=o, dotsize=0.4, fillcolor=white](3.5333333,-4.62)
			\psdots[linecolor=black, fillstyle=solid, dotstyle=o, dotsize=0.4, fillcolor=white](2.7333333,-4.62)
			\psdots[linecolor=black, fillstyle=solid, dotstyle=o, dotsize=0.4, fillcolor=white](9.933333,-4.62)
			\psdots[linecolor=black, dotsize=0.1](6.733333,-4.62)
			\psdots[linecolor=black, dotsize=0.1](7.133333,-4.62)
			\psdots[linecolor=black, dotsize=0.1](7.5333333,-4.62)
			\psdots[linecolor=black, dotsize=0.4](5.9333334,-4.62)
			\psdots[linecolor=black, fillstyle=solid, dotstyle=o, dotsize=0.4, fillcolor=white](7.5333333,-2.22)
			\psdots[linecolor=black, fillstyle=solid, dotstyle=o, dotsize=0.4, fillcolor=white](6.733333,-1.42)
			\psdots[linecolor=black, fillstyle=solid, dotstyle=o, dotsize=0.4, fillcolor=white](5.9333334,-0.62)
			\psdots[linecolor=black, fillstyle=solid, dotstyle=o, dotsize=0.4, fillcolor=white](3.5333333,-11.82)
			\psdots[linecolor=black, dotsize=0.4](18.733334,-5.42)
			\psdots[linecolor=black, dotsize=0.4](17.933332,-5.42)
			\psdots[linecolor=black, dotsize=0.4](17.133333,-5.42)
			\psdots[linecolor=black, dotsize=0.4](16.333334,-5.42)
			\psdots[linecolor=black, dotsize=0.4](18.733334,-6.22)
			\psdots[linecolor=black, dotsize=0.4](17.933332,-6.22)
			\psdots[linecolor=black, dotsize=0.4](17.133333,-6.22)
			\psdots[linecolor=black, dotsize=0.4](16.333334,-6.22)
			\psdots[linecolor=black, dotsize=0.4](17.933332,-7.02)
			\psdots[linecolor=black, dotsize=0.4](17.133333,-7.02)
			\psdots[linecolor=black, dotsize=0.4](16.333334,-7.02)
			\psdots[linecolor=black, fillstyle=solid, dotstyle=o, dotsize=0.4, fillcolor=white](15.533334,-5.42)
			\psdots[linecolor=black, fillstyle=solid, dotstyle=o, dotsize=0.4, fillcolor=white](15.533334,-6.22)
			\psdots[linecolor=black, fillstyle=solid, dotstyle=o, dotsize=0.4, fillcolor=white](15.533334,-7.02)
			\psline[linecolor=black, linewidth=0.08](2.3333333,-9.42)(2.3333333,-7.02)(2.3333333,-7.02)
			\psline[linecolor=black, linewidth=0.08](19.133333,-9.02)(2.3333333,-9.02)(2.3333333,-9.02)
			\psline[linecolor=black, linewidth=0.08](2.3333333,-9.82)(19.133333,-9.82)(19.133333,-9.82)
			\psline[linecolor=black, linewidth=0.08](2.3333333,-9.02)(2.3333333,-10.62)(2.3333333,-11.42)(2.3333333,-11.42)
			\psline[linecolor=black, linewidth=0.08](19.133333,-8.22)(19.133333,-10.62)(19.133333,-10.62)
			\psdots[linecolor=black, dotsize=0.1](4.3333335,-10.22)
			\psdots[linecolor=black, dotsize=0.1](4.3333335,-10.62)
			\psdots[linecolor=black, dotsize=0.1](4.3333335,-11.02)
			\psdots[linecolor=black, dotsize=0.4](18.733334,-9.42)
			\psdots[linecolor=black, dotsize=0.4](17.933332,-9.42)
			\psdots[linecolor=black, dotsize=0.4](17.133333,-9.42)
			\psdots[linecolor=black, dotsize=0.4](16.333334,-9.42)
			\psdots[linecolor=black, dotsize=0.4](18.733334,-11.82)
			\psdots[linecolor=black, dotsize=0.4](17.933332,-11.82)
			\psdots[linecolor=black, dotsize=0.4](17.133333,-11.82)
			\psdots[linecolor=black, dotsize=0.1](15.933333,-11.82)
			\psdots[linecolor=black, dotsize=0.1](15.533334,-11.82)
			\psdots[linecolor=black, dotsize=0.1](15.133333,-11.82)
			\psdots[linecolor=black, dotsize=0.4](13.933333,-11.82)
			\psdots[linecolor=black, dotsize=0.4](14.733334,-11.82)
			\psdots[linecolor=black, dotsize=0.4](16.333334,-11.82)
			\psdots[linecolor=black, fillstyle=solid, dotstyle=o, dotsize=0.4, fillcolor=white](13.133333,-11.82)
			\psline[linecolor=black, linewidth=0.08](6.3333335,-11.42)(6.3333335,-12.22)(6.3333335,-12.22)
			\psline[linecolor=black, linewidth=0.08](12.733334,-11.42)(12.733334,-12.22)(12.733334,-12.22)
			\psdots[linecolor=black, fillstyle=solid, dotstyle=o, dotsize=0.4, fillcolor=white](5.9333334,-11.82)
			\psdots[linecolor=black, fillstyle=solid, dotstyle=o, dotsize=0.4, fillcolor=white](15.533334,-9.42)
			\psdots[linecolor=black, fillstyle=solid, dotstyle=o, dotsize=0.4, fillcolor=white](9.133333,-9.42)
			\psdots[linecolor=black, dotsize=0.4](8.333333,-9.42)
			\psdots[linecolor=black, dotsize=0.1](7.133333,-9.42)
			\psdots[linecolor=black, dotsize=0.1](6.733333,-9.42)
			\psdots[linecolor=black, dotsize=0.1](7.5333333,-9.42)
			\psdots[linecolor=black, dotsize=0.4](5.9333334,-9.42)
			\psdots[linecolor=black, dotsize=0.4](5.133333,-9.42)
			\psdots[linecolor=black, dotsize=0.4](4.3333335,-9.42)
			\psdots[linecolor=black, fillstyle=solid, dotstyle=o, dotsize=0.4, fillcolor=white](3.5333333,-9.42)
			\psdots[linecolor=black, fillstyle=solid, dotstyle=o, dotsize=0.4, fillcolor=white](2.7333333,-9.42)
			\psdots[linecolor=black, fillstyle=solid, dotstyle=o, dotsize=0.4, fillcolor=white](3.5333333,-7.02)
			\psdots[linecolor=black, fillstyle=solid, dotstyle=o, dotsize=0.4, fillcolor=white](2.7333333,-7.02)
			\psdots[linecolor=black, fillstyle=solid, dotstyle=o, dotsize=0.4, fillcolor=white](3.5333333,-6.22)
			\psdots[linecolor=black, fillstyle=solid, dotstyle=o, dotsize=0.4, fillcolor=white](2.7333333,-6.22)
			\psdots[linecolor=black, fillstyle=solid, dotstyle=o, dotsize=0.4, fillcolor=white](2.7333333,-5.42)
			\psdots[linecolor=black, fillstyle=solid, dotstyle=o, dotsize=0.4, fillcolor=white](3.5333333,-5.42)
			\psdots[linecolor=black, fillstyle=solid, dotstyle=o, dotsize=0.4, fillcolor=white](5.9333334,-5.42)
			\psdots[linecolor=black, fillstyle=solid, dotstyle=o, dotsize=0.4, fillcolor=white](6.733333,-6.22)
			\psdots[linecolor=black, fillstyle=solid, dotstyle=o, dotsize=0.4, fillcolor=white](7.5333333,-7.02)
			\psdots[linecolor=black, dotsize=0.4](4.3333335,-5.42)
			\psdots[linecolor=black, dotsize=0.4](5.133333,-5.42)
			\psdots[linecolor=black, dotsize=0.4](4.3333335,-6.22)
			\psdots[linecolor=black, dotsize=0.4](5.133333,-6.22)
			\psdots[linecolor=black, dotsize=0.4](5.9333334,-6.22)
			\psdots[linecolor=black, dotsize=0.4](4.3333335,-7.02)
			\psdots[linecolor=black, dotsize=0.4](5.133333,-7.02)
			\psdots[linecolor=black, dotsize=0.4](5.9333334,-7.02)
			\psdots[linecolor=black, dotsize=0.4](6.733333,-7.02)
			\psdots[linecolor=black, dotsize=0.1](10.733334,-7.82)
			\psdots[linecolor=black, dotsize=0.1](10.733334,-8.22)
			\psdots[linecolor=black, dotsize=0.1](10.733334,-8.62)
			\psline[linecolor=black, linewidth=0.08](6.3333335,-5.02)(6.3333335,-5.82)(7.133333,-5.82)(7.133333,-6.62)(7.9333334,-6.62)(7.9333334,-7.42)(7.9333334,-7.42)
			\psline[linecolor=black, linewidth=0.08](9.533334,-9.02)(9.533334,-9.82)(9.533334,-9.82)
			\psline[linecolor=black, linewidth=0.08](15.133333,-5.02)(15.133333,-7.42)(15.133333,-7.42)
			\psline[linecolor=black, linewidth=0.08](15.133333,-9.02)(15.133333,-9.82)(15.133333,-9.82)
			\rput[bl](8.786667,9.326667){$v_{i-2}$}
			\rput[bl](13.373333,9.246667){$v_{n-i+3}$}
			\rput[bl](1.5333333,0.033333335){i-1}
			\rput[bl](1.7333333,-0.7266667){i}
			\rput[bl](1.3333334,-1.6066667){i+1}
			\rput[bl](1.3199999,-2.3533332){i+2}
			\rput[bl](0.0,-12.073334){$\frac{(i-4)(i-5)}{2}+3$}
			\rput[bl](1.4,-4.7266665){2i-6}
			\rput[bl](1.4133333,-5.5){2i-5}
			\rput[bl](1.4,-6.34){2i-4}
			\rput[bl](1.4,-7.14){2i-3}
			\rput[bl](1.3066666,-9.513333){3i-12}
			\psdots[linecolor=black, dotsize=0.4](9.133333,0.18)
			\psdots[linecolor=black, dotsize=0.4](8.333333,-4.62)
			\end{pspicture}
		}
	\end{center}
	\caption{\footnotesize Making isolate dominating sets of $C_n$ such that $v_1,v_2 \not\in S$, related to the proof of Theorem \ref{cycle-thm}} \label{cycle2}
\end{figure}

Here, we consider the number of isolate dominating sets of cycle $C_n$. 

\begin{theorem}\label{cycle-thm}
	The number of isolate dominating sets of Cycle $C_n$ with cardinality $i$, $d_0(C_n,i)$, is: 
	\begin{align*}
	d_0(C_n,i) & =   2d(P_{n-3},i-1)+d(P_{n-6},i-2)+2d(P_{n-6},i-3) \\
	&\quad +\displaystyle\sum_{k=3}^{i}k d_0(P_{n-k-1},i-k+1)+\displaystyle\sum_{k=4}^{i-1}(k-3) d_0(P_{n-k-4},i-k).
	\end{align*}
\end{theorem}

\begin{proof}
	Let $V(C_n)=\{v_1,v_2,....,v_n\}$  and  $S$ be an isolate  dominating  set of $C_n$ with cardinality $i$. First we consider dominating sets of $C_n$ such that at least one of the vertices $v_1$ and $v_2$ are in $S$ (see Figure \ref{cycle}). We have following cases: 
	
	\medskip			
	\noindent {\bf Case (1)} 
	If $v_1\in S$ and $v_2,v_n\not\in S$, then ${\cal D}_0(C_n,i)={\cal D}(P_{n-3},i-1)\cup \{v_1\}$. In this case the number of isolate dominating sets of $C_n$ with cardinality $i$ is $d(P_{n-3},i-1)$.  
	
	\medskip			
	\noindent {\bf Case (2)} 
	If $v_1\not\in S, v_2\in S$ and $v_3\not\in S$, then ${\cal D}_0(C_n,i)={\cal D}(P_{n-3},i-1)\cup \{v_2\}$. In this case the number of isolate dominating sets of $C_n$ with cardinality $i$ is $d(P_{n-3},i-1)$.  
	
	\medskip			
	\noindent {\bf Case (3)} 
	If $v_1,v_2\in S,$ and $v_3,v_n\not\in S$, then ${\cal D}_0(C_n,i)={\cal D}_0(P_{n-4},i-2)\cup \{v_1,v_2\}$. In this case the number of isolate dominating sets of $C_n$ with cardinality $i$ is $d_0(P_{n-4},i-2)$. 
	
	\medskip			
	\noindent {\bf Case (4)} 
	If $v_1,v_n\in S,$ and $v_2,v_{n-1}\not\in S$, then ${\cal D}_0(C_n,i)={\cal D}_0(P_{n-4},i-2)\cup \{v_1,v_n\}$. In this case the number of isolate dominating sets of $C_n$ with cardinality $i$ is $d_0(P_{n-4},i-2)$. 		
	
	\medskip			
	\noindent {\bf Case (5)} 
	If $v_2,v_3\in S,$ and $v_1,v_4\not\in S$, then ${\cal D}_0(C_n,i)={\cal D}_0(P_{n-4},i-2)\cup \{v_2,v_3\}$. In this case the number of isolate dominating sets of $C_n$ with cardinality $i$ is $d_0(P_{n-4},i-2)$. 
	
	\medskip			
	\noindent {\bf Case (6)} 
	If $v_1,v_2,v_3\in S$ and $v_4,v_n\not\in S$, then ${\cal D}_0(C_n,i)={\cal D}_0(P_{n-5},i-3)\cup \{v_1,v_2,v_3\}$. In this case $d_0(C_n,i)$ is $d_0(P_{n-5},i-3)$. 
	
	\medskip			
	\noindent {\bf Case (7)} 
	If $v_1,v_2,v_n\in S$ and $v_3,v_{n-1}\not\in S$, then ${\cal D}_0(C_n,i)={\cal D}_0(P_{n-5},i-3)\cup \{v_1,v_2,v_3\}$. In this case the number of isolate dominating sets of $C_n$ with cardinality $i$ is $d_0(P_{n-5},i-3)$. 
	
	\medskip			
	\noindent {\bf Case (8)} 
	If $v_1,v_{n-1},v_n\in S$ and $v_2,v_{n-2}\not\in S$, then ${\cal D}_0(C_n,i)={\cal D}_0(P_{n-5},i-3)\cup \{v_1,v_2,v_3\}$. In this case the number of isolate dominating sets of $C_n$ with cardinality $i$ is $d_0(P_{n-5},i-3)$. 
	
	\medskip
	\noindent {\bf Case (9)}
	If $v_2,v_3,v_4\in S,$ and $v_1,v_5\not\in S$, then ${\cal D}_0(C_n,i)={\cal D}_0(P_{n-5},i-3)\cup \{v_1,v_2,v_3\}$. In this case $d_0(C_n,i)$ is $d_0(P_{n-5},i-3)$. 												
	
	\medskip
	By continuing these steps we will have the following $i+1$ end steps: 
	
	\medskip
	\noindent {\bf  Case ($\frac{i(i+1)}{2}$)}
	If $v_1,v_2,v_3,\ldots,v_{i-1}\in S$ and $v_i,v_n\not\in S$, then
	$${\cal D}_0(C_n,i)={\cal D}_0(P_{n-i-1},1)\cup \{v_1,v_2,v_3,\ldots,v_{i-1}\}.$$
	In this case the number of isolate dominating sets of $C_n$ with cardinality $i$ is $d_0(P_{n-i},1)$. 	
	
	\medskip
	\noindent {\bf Case ($\frac{i(i+1)}{2}+1$)}
	If $v_n,v_1,v_2,v_3,\ldots,v_{i-2}\in S,$ and $v_i,v_{n-1}\not\in S$, then 
	$${\cal D}_0(C_n,i)={\cal D}_0(P_{n-i-1},1)\cup \{v_n,v_1,v_2,v_3,\ldots,v_{i-2}\}.$$
	In this case the number of isolate dominating sets of $C_n$ with cardinality $i$ is $d_0(P_{n-i-1},1)$. 
	
	\medskip
	\noindent {\bf Case($\frac{(i+2)(i+1)}{2}-2$)}
	If $v_{n-i+2},v_{n-i+3},\ldots,v_n,v_1\in S,$ and $v_2,v_{n-i+1}\not\in S$, then 
	$${\cal D}_0(C_n,i)={\cal D}_0(P_{n-i-1},1)\cup \{v_{n-i+2},v_{n-i+3},\ldots,v_n,v_1\}.$$
	In this case the number of isolate dominating sets of $C_n$ with cardinality $i$ is $d_0(P_{n-i-1},1)$.
	\medskip	
	\noindent {\bf Case ($\frac{(i+2)(i+1)}{2}-1$)} 
	If $v_2,v_3,v_4,\ldots,v_i\in S,$ and $v_1,v_{i+1}\not\in S$, then 
	$${\cal D}_0(C_n,i)={\cal D}_0(P_{n-i-1},1)\cup \{v_2,v_3,v_4,\ldots,v_i\}.$$
	In this case the number of isolate dominating sets of $C_n$ with cardinality $i$ is $d_0(P_{n-i-1},1)$. 			

	Now we consider dominating sets of $C_n$ such that $v_1,v_2 \not\in S$ (see Figure \ref{cycle2}). Therefore we should have $v_3,v_n \in S$. So we have following cases:

	\medskip			
	\noindent {\bf Case ($1^\prime$)} 
	If $v_3,v_n\in S$ and $v_4,v_{n-1}\not\in S$, then ${\cal D}_0(C_n,i)={\cal D}(P_{n-6},i-2)\cup \{v_3,v_n\}$. In this case the number of isolate dominating sets of $C_n$ with cardinality $i$ is $d(P_{n-6},i-2)$.  
	
	\medskip			
	\noindent {\bf Case ($2'$)} 
	If $v_3,v_{n-1},v_n\in S$ and $v_4\not\in S$, then ${\cal D}_0(C_n,i)={\cal D}(P_{n-6},i-3)\cup \{v_3,v_{n-1},v_n\}$. In this case the number of isolate dominating sets of $C_n$ with cardinality $i$ is $d(P_{n-6},i-3)$.  
	
	\medskip			
	\noindent {\bf Case ($3'$)} 
	If $v_3,v_4,v_n\in S$ and $v_{n-1}\not\in S$, then ${\cal D}_0(C_n,i)={\cal D}(P_{n-6},i-3)\cup \{v_3,v_4,v_n\}$. In this case the number of isolate dominating sets of $C_n$ with cardinality $i$ is $d(P_{n-6},i-3)$.  
	
	\medskip			
	\noindent {\bf Case ($4'$)} 
	If $v_3,v_4,v_{n-1},v_n\in S$ and $v_5,v_{n-2}\not\in S$, then ${\cal D}_0(C_n,i)={\cal D}_0(P_{n-8},i-4)\cup \{v_3,v_4,v_{n-1},v_n\}$. In this case the number of isolate dominating sets of $C_n$ with cardinality $i$ is $d_0(P_{n-8},i-4)$. 		
	
	\medskip			
	\noindent {\bf Case ($5'$)} 
	If $v_3,v_4,v_5,v_{n-1},v_n\in S$ and $v_6,v_{n-2}\not\in S$, then ${\cal D}_0(C_n,i)={\cal D}_0(P_{n-9},i-5)\cup \{v_3,v_4,v_5,v_{n-1},v_n\}$. In this case the number of isolate dominating sets of $C_n$ with cardinality $i$ is $d_0(P_{n-9},i-5)$. 
	
	\medskip			
	\noindent {\bf Case ($6'$)} 
	If $v_3,v_4,v_5,v_6,v_{n-1},v_n\in $ and $v_7,v_{n-2}\not\in S$, then ${\cal D}_0(C_n,i)={\cal D}_0(P_{n-10},i-6)\cup \{v_3,v_4,v_5,v_6,v_{n-1},v_n\}$. In this case the number of isolate dominating sets of $C_n$ with cardinality $i$ is $d_0(P_{n-10},i-6)$. 
	
	\medskip			
	\noindent {\bf Case ($i-1'$)} 
	If $v_3,v_4,v_5,\ldots,v_{i-1},v_{n-1},v_n\in S $ and $v_i,v_{n-2}\not\in S$, then ${\cal D}_0(C_n,i)={\cal D}_0(P_{n-i-3},1)\cup \{v_3,v_4,v_5,\ldots,v_{i-1},v_{n-1},v_n\}$. In this case the number of isolate dominating sets of $C_n$ with cardinality $i$ is $d_0(P_{n-i-3},1)$. 
	
	\medskip			
	\noindent {\bf Case ($i'$)} 
	If $v_3,v_4,v_{n-1},v_{n-2},v_n\in S$ and $v_5,v_{n-3}\not\in S$, then ${\cal D}_0(C_n,i)={\cal D}_0(P_{n-9},i-5)\cup \{v_3,v_4,v_{n-1},v_{n-2},v_n\}$. In this case the number of isolate dominating sets of $C_n$ with cardinality $i$ is $d_0(P_{n-9},i-5)$. 		
	
	\medskip			
	\noindent {\bf Case ($i+1'$)} 
	If $v_3,v_4,v_5,v_{n-2},v_{n-1},v_n\in S$ and $v_6,v_{n-3}\not\in S$, then ${\cal D}_0(C_n,i)={\cal D}_0(P_{n-10},i-6)\cup \{v_3,v_4,v_5,v_{n-2},v_{n-1},v_n\}$. In this case the number of isolate dominating sets of $C_n$ with cardinality $i$ is $d_0(P_{n-10},i-6)$. 
	
	\medskip			
	\noindent {\bf Case ($i+2'$)} 
	If $v_3,v_4,v_5,v_6,v_{n-2},v_{n-1},v_n\in $ and $v_7,v_{n-3}\not\in S$, then ${\cal D}_0(C_n,i)={\cal D}_0(P_{n-11},i-7)\cup \{v_3,v_4,v_5,v_6,v_{n-2},v_{n-1},v_n\}$. In this case the number of isolate dominating sets of $C_n$ with cardinality $i$ is $d_0(P_{n-11},i-7)$.

	\medskip			
	\noindent {\bf Case ($2i-6'$)} 
	If $v_3,v_4,v_5,\ldots,v_{i-2},v_{n-1},v_n\in S $ and $v_{i-1},v_{n-3}\not\in S$, then ${\cal D}_0(C_n,i)={\cal D}_0(P_{n-i-3},1)\cup \{v_3,v_4,v_5,\ldots,v_{i-2},v_{n-1},v_n\}$. In this case the number of isolate dominating sets of $C_n$ with cardinality $i$ is $d_0(P_{n-i-3},1)$. 											
	
	\medskip
	By continuing these steps we will have the following end step: 
	
	\medskip
	\noindent {\bf  Case ($\frac{(i-4)(i-5)}{2}+3'$)}
If $v_3,v_4,v_{n-i+3},v_{n-i+4}\ldots,v_{n-1},v_n\in S $ and $v_5,v_{n-i+2}\not\in S$, then ${\cal D}_0(C_n,i)={\cal D}_0(P_{n-i-3},1)\cup \{v_3,v_4,v_{n-i+3},v_{n-i+4}\ldots,v_{n-1},v_n\}$. In this case the number of isolate dominating sets of $C_n$ with cardinality $i$ is $d_0(P_{n-i-3},1)$.

	\medskip
	So we have:
	\begin{align*}
	d_0(C_n,i) & =  d(P_{n-3},i-1) + d(P_{n-3},i-1) \\
	&\quad + d_0(P_{n-4},i-2) +d_0(P_{n-4},i-2) 
	+ d_0(P_{n-4},i-2)\\
	&\quad + d_0(P_{n-5},i-3)+d_0(P_{n-5},i-3)
	+ d_0(P_{n-5},i-3)+d_0(P_{n-5},i-3)\\
	&\quad  + \ldots  \\
	&\quad  + d_0(P_{n-i-1},1)+ d_0(P_{n-i-1},1)+\ldots + d_0(P_{n-i-1},1) \\
	&\quad  +d(P_{n-6},i-2)+d(P_{n-6},i-3)+d(P_{n-6},i-3)\\
	&\quad  +d_0(P_{n-8},i-4)+d_0(P_{n-9},i-5)+d_0(P_{n-10},i-6)+\ldots +d_0(P_{n-i-3},1)\\
	&\quad  +d_0(P_{n-9},i-5)+d_0(P_{n-10},i-6)+d_0(P_{n-11},i-7)+\ldots +d_0(P_{n-i-3},1)\\
	&\quad +\ldots \\
	&\quad +d_0(P_{n-i-3},1)\\
	&= 2d(P_{n-3},i-1)+d(P_{n-6},i-2)+2d(P_{n-6},i-3) \\
	&\quad +\displaystyle\sum_{k=3}^{i}k d_0(P_{n-k-1},i-k+1)+\displaystyle\sum_{k=4}^{i-1}(k-3) d_0(P_{n-k-4},i-k),
	\end{align*}
	and therefore we have the result.\qed	
\end{proof}

\section{The number of I.D. sets in corona product}

In this section we study the number of isolate dominating (I.D.) sets of corona product  of two graphs. First, we consider the centipede graph $P_n\circ K_1$. It is easy to see that $\gamma_0(P_n\circ K_1)=n$.

	\begin{theorem} \label{cendipate}
 	The number of isolate dominating sets of  $P_n\circ K_1$ with cardinality $n+i$, for $0 \leq i \leq  n-2 $, $d_0(P_n\circ K_1,i)$, is:
 	$$d_0(P_n\circ K_1,n+i)=\sum_{k=1}^{n-i} {n \choose i+k}{i+k \choose i}.$$
	\end{theorem}

	\begin{proof}
	Consider the graph $P_n\circ K_1$ in Figure \ref{cendipatethm}. To construct an isolate dominating set $S$ of $P_n\circ K_1$,   we choose $i+k$ ($0 \leq i \leq  n-2 $) leaves from the set of leaves of graph (such as  $u_{t_1},u_{t_2},u_{t_3},\ldots,u_{t_{i+k}}$, ($ 1\leq k \leq n-i $)), and $n-i-k$  vertices from the  set $\{v_1,v_2,\ldots,v_n\}$, which are not the neighbour of chosen leaves (such as $\{v_{t_{i+k+1}}, v_{t_{i+k+2}},...,v_{t_n}\}$). So the set $S$ is an isolate dominating set with cardinality $n$. To have isolate dominating set of cardinality $n+i$ ($0\leq i \leq n-2$),  we choose only $i$ vertices  from the neighbours of the chosen leaves (i.e. $u_{t_1},u_{t_2},u_{t_3},\ldots,u_{t_{i+k}}$). Since there is at least one vertex from  the set $\{u_{t_1},u_{t_2},u_{t_3},\ldots,u_{t_{i+k}}\}$ which is not in the dominating set, therefore we have the result. 
	\qed	
	\end{proof}

		\begin{figure}
		\begin{center}
		\psscalebox{0.8 0.8}
{
\begin{pspicture}(0,-3.385)(8.861389,-0.955)
\psline[linecolor=black, linewidth=0.08](0.20138885,-1.585)(4.201389,-1.585)(3.8013887,-1.585)
\psline[linecolor=black, linewidth=0.08](5.8013887,-1.585)(7.4013886,-1.585)(8.601389,-1.585)(8.601389,-1.585)
\psline[linecolor=black, linewidth=0.08](0.20138885,-1.585)(0.20138885,-2.785)(0.20138885,-2.785)
\psline[linecolor=black, linewidth=0.08](1.4013889,-1.585)(1.4013889,-2.785)(1.4013889,-2.785)
\psline[linecolor=black, linewidth=0.08](2.601389,-1.585)(2.601389,-2.785)(2.601389,-2.785)
\psline[linecolor=black, linewidth=0.08](3.8013887,-1.585)(3.8013887,-2.785)(3.8013887,-2.785)
\psline[linecolor=black, linewidth=0.08](6.201389,-1.585)(6.201389,-2.785)(6.201389,-2.785)
\psline[linecolor=black, linewidth=0.08](7.4013886,-1.585)(7.4013886,-2.785)(7.4013886,-2.785)
\psline[linecolor=black, linewidth=0.08](8.601389,-1.585)(8.601389,-2.785)(8.601389,-2.785)
\psdots[linecolor=black, dotstyle=o, dotsize=0.4, fillcolor=white](0.20138885,-1.585)
\psdots[linecolor=black, dotstyle=o, dotsize=0.4, fillcolor=white](1.4013889,-1.585)
\psdots[linecolor=black, dotstyle=o, dotsize=0.4, fillcolor=white](2.601389,-1.585)
\psdots[linecolor=black, dotstyle=o, dotsize=0.4, fillcolor=white](3.8013887,-1.585)
\psdots[linecolor=black, dotstyle=o, dotsize=0.4, fillcolor=white](6.201389,-1.585)
\psdots[linecolor=black, dotstyle=o, dotsize=0.4, fillcolor=white](7.4013886,-1.585)
\psdots[linecolor=black, dotstyle=o, dotsize=0.4, fillcolor=white](8.601389,-1.585)
\psdots[linecolor=black, dotstyle=o, dotsize=0.4, fillcolor=white](0.20138885,-2.785)
\psdots[linecolor=black, dotstyle=o, dotsize=0.4, fillcolor=white](1.4013889,-2.785)
\psdots[linecolor=black, dotstyle=o, dotsize=0.4, fillcolor=white](2.601389,-2.785)
\psdots[linecolor=black, dotstyle=o, dotsize=0.4, fillcolor=white](3.8013887,-2.785)
\psdots[linecolor=black, dotstyle=o, dotsize=0.4, fillcolor=white](6.201389,-2.785)
\psdots[linecolor=black, dotstyle=o, dotsize=0.4, fillcolor=white](7.4013886,-2.785)
\psdots[linecolor=black, dotstyle=o, dotsize=0.4, fillcolor=white](8.601389,-2.785)
\psdots[linecolor=black, dotsize=0.1](4.601389,-1.585)
\psdots[linecolor=black, dotsize=0.1](5.001389,-1.585)
\psdots[linecolor=black, dotsize=0.1](5.4013886,-1.585)
\rput[bl](0.0013888549,-1.225){$v_1$}
\rput[bl](1.1813889,-1.245){$v_2$}
\rput[bl](2.4213889,-1.225){$v_3$}
\rput[bl](3.581389,-1.265){$v_4$}
\rput[bl](5.8013887,-1.225){$v_{n-2}$}
\rput[bl](7.041389,-1.245){$v_{n-1}$}
\rput[bl](8.461389,-1.265){$v_n$}
\rput[bl](0.061388854,-3.385){$u_1$}
\rput[bl](1.2613889,-3.385){$u_2$}
\rput[bl](2.4213889,-3.365){$u_3$}
\rput[bl](3.601389,-3.365){$u_4$}
\rput[bl](8.381389,-3.325){$u_n$}
\rput[bl](5.8813887,-3.345){$u_{n-2}$}
\rput[bl](7.081389,-3.365){$u_{n-1}$}
\end{pspicture}
}
		\end{center}
		\caption{\small Graph $P_n\circ K_1$ related to the proof of Theorem \ref{cendipate}} \label{cendipatethm}
	\end{figure}
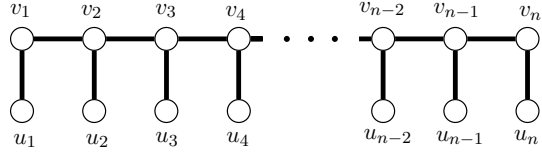	
	
	As an immediate result of the Theorem \ref{cendipate}, we have:

	\begin{corollary}
  For every $n\in \mathbb{N}$, $d_0(P_n\circ K_1,n)=2^n-1.$
 	 \end{corollary}
	
		\begin{theorem}\label{cen2n-1}
For every $n\in \mathbb{N}$,   $d_0(P_n\circ K_1,2n-1)=n.$
	\end{theorem}
	
		\begin{proof}
	It suffices to put vertices $u_1,u_2,u_3,\ldots,u_n$ (see Figure \ref{cendipatethm}) in the dominating  set and choose $n-1$ vertices from $v_1,v_2,v_3,\ldots,v_n$ and put them into this set.
	\qed	
	\end{proof}
	
By similar argument to  Theorem \ref{cendipate}, we have the following result:	
	
		\begin{theorem}
 Let $G$ be a graph of order $n$. For every  $0 \leq i \leq n-2$,
 	$$d_0(G\circ K_1,n+i)=\sum_{k=1}^{n-i} {n \choose i+k}{i+k \choose i}.$$
	\end{theorem}
	

By similar argument to  Theorem \ref{cen2n-1}, we have the following result:	
	
		\begin{theorem}
 Let $G$ be a graph of order $n$. Then:
 	$$d_0(G\circ K_1,2n-1)=n.$$
	\end{theorem}

		\begin{theorem}
	The generating function for the number of isolate  dominating sets of  $G\circ K_1$ is:
 	$$D_0(G\circ K_1,x)=nx^{2n-1}+\displaystyle\sum_{j=1}^n {n\choose j}(2^j-1)x^{2n-j}.$$
	\end{theorem} 
	\begin{proof}
		By the definition of isolate domination polynomial,  we have 
		\begin{eqnarray*}
			D_0(G\circ K_1,x)&=&nx^{2n-1}+\displaystyle \sum_{i=0}^{n-2}\sum_{k=1}^{n-i} {n\choose i+k}
		{i+k\choose i}x^{n+i}\\	
	&=&nx^{2n-1}+\sum_{i=0}^{n-2}\sum_{k=1}^{n-i} {n\choose i}x^i {n-i\choose k}x^n.
\end{eqnarray*}
By simple computation we see that 
$$\sum_{i=0}^{n-2}\sum_{k=1}^{n-i} {n\choose i}x^i {n-i\choose k}x^n=\displaystyle\sum_{j=1}^{n}{n\choose j}(2^j-1)x^{2n-j}.$$ 		
		Therefore we have the result. \qed
	 \end{proof} 

Now we consider the graph $K_1\circ G$.
	\begin{theorem}\label{KH}
		 For every graph $G$ of order $n$, 
		\begin{enumerate} 
			\item[(i)]  $\gamma_0(K_1\circ G)=1$.  
			\item[(ii)] $d_0(K_1\circ G,1)=1+d_0(G,1)$ and for $i\geq \gamma_0(G)$, 
			$d_0(K_1\circ G,i)=d_0(G,i).$
			\item[(iii)] $D_0(K_1\circ G,x)=(1+d_0(G,1))x+D_0(G,x)$.  
			\end{enumerate} 
	\end{theorem}
		\begin{proof}
			\begin{enumerate} 
				\item[(i)]  In the graph $K_1\circ G$, it is suffices to consider the vertex $K_1$ (which is adjacent to all vertices of $G$) as an isolate dominating set. 
				
				\item[(ii)] The only  isolate dominating sets of $K_1\circ G$ with cardinality one  is the vertex $K_1$ (which is adjacent to all vertices of $G$). So $d_0(K_1\circ G,1)=1$. Since any isolate dominating set of $G$ in $K_1\circ G$ is isolate dominating set of $K_1\circ G$, so for $i\geq \gamma_0(G)$, 
				$d_0(K_1\circ G,i)=d_0(G,i).$ 
				\item[(iii)] It follows from the definition of isolate domination polynomial. \qed 
			\end{enumerate} 
				\end{proof}

The following theorem gives an upper bound for the number of isolate dominating sets of $G\circ H$:

\begin{theorem} \label{goh2}
	Let $G$ and $H$ be two graphs of order $n$ and $m$, respectively. Then for $1\leq t \leq n$, 
	$$d_0(G\circ H,t\gamma_0(H)+n-t)\leq t{n \choose t}d_0(H,\gamma_0(H))(d(H,\gamma_0(H)))^{t-1}.$$
	\end{theorem} 
	
	\begin{proof} 
	 We construct an isolate dominating set $S$ of $G\circ H$. Suppose that  $S\cap V(G)=S_G$ and $|S_G|=n-t$. So $t$ vertices of $G$ should be 
	 dominate with vertices of their correspond copies of $H$, say $H_1,...,H_t$. We can choose these $t$ vertices by ${n\choose t}$ cases. 
	 Consider one of these copies of $H$, say $H_1$  and suppose that $S_1$ is  an isolate dominating set of $H_1$  with cardinality $\gamma_0(H)$. 
	 Also suppose that $S_i$ ($2\leq i\leq n$) is  dominating sets of $H_i$ ($2\leq i\leq n$). Then $\displaystyle\cup_{i=1}^n S_i\cup S_G$ is an isolate dominating set of $G\circ H$ with cardinality $t\gamma_0(H)+n-t$. The number of isolate dominating sets of the form $\cup_{i=1}^n S_i\cup S_G$ is ${n \choose t}{t \choose 1}d_0(H,\gamma_0(H))(d(H,\gamma_0(H)))^{t-1}$. Therefore we have the result.\qed
\end{proof}

\begin{remark}
The upper bounds in Theorem  \ref{goh2} is sharp. It suffices to consider $t=1$.
\end{remark}

\section{Acknowledgements} 

The first author would like to thank the Research Council of Norway (NFR Toppforsk Project Number 274526, Parameterized Complexity for Practical Computing) and Department of Informatics, University of
Bergen for their support. Also he is thankful to Michael Fellows and
Michal Walicki for conversations and appreciation to them and Frances
Rosamond for sharing their pearls of wisdom with him during the course
of this research.

\end{document}